\documentclass[12pt]{amsart}
\usepackage{amsmath,amssymb}

\theoremstyle{plain}
\newtheorem{thm}{Theorem}[section]
\newtheorem{theorem}[thm]{Theorem}

\newtheorem{lemma}[thm]{Lemma}
\newtheorem{corollary}[thm]{Corollary}
\newtheorem{proposition}[thm]{Proposition}
\theoremstyle{definition}
\newtheorem{remark}[thm]{Remark}

\newtheorem{notation}[thm]{Notation}
\newtheorem{definition}[thm]{Definition}
\newtheorem{condition}[thm]{Condition}

\newtheorem{example}[thm]{Example}
\newtheorem{examples}[thm]{Examples}

\numberwithin{equation}{section}
\newcommand{\0}{{\mathcal O}}

\newcommand{\sC}{{\mathcal C}}

\newcommand{\sI}{{\mathcal I}}

\newcommand{\sO}{{\mathcal O}}

\newcommand{\sZ}{{\mathcal Z}}

\newcommand{\C}{{\mathbb C}}

\newcommand{\BP}{{\mathbb P}}
\newcommand{\pit}{{\mathbb P}}
\newcommand{\Q}{{\mathbb Q}}

\newcommand{\End}{{\rm End}}

\def\cit{{\mathbb C}}

\def\qit{{\mathbb Q}}
\def\pit{{\mathbb P}}
\def\BP{{\mathbb P}}

\newcommand{\fg}{{\mathfrak g}}

\newcommand{\aut}{{\mathfrak a}{\mathfrak u}{\mathfrak t}}

\def\Gr{\mathop{\rm Gr}\nolimits}

\def\Lag{\mathop{\rm Lag}\nolimits}
\def\Sym{\mathop{\rm Sym}\nolimits}

\def\Hom{\mathop{\rm Hom}\nolimits}

\def\Bl{\mathop{\rm Bl}\nolimits}

\def\SYM{\mathop{\rm Sym}\nolimits}

\title[Special birational transformations of type $(2,1)$
]{Special birational transformations of type $(2,1)$}
\author{Baohua Fu and Jun-Muk Hwang}
\thanks{Baohua Fu is supported by National Scientific Foundation of
China (11225106 and 11321101).}

\begin{document}
\maketitle

\begin{abstract}
A birational transformation
 $\Phi: \BP^n \dasharrow Z \subset \BP^N,$ where $Z \subset \BP^N$ is a nonsingular variety of Picard number 1, is called  a {\em special birational transformation of type $(a,b)$} if $\Phi$ is given by a linear system of degree $a$, its inverse  $\Phi^{-1}$ is given by a linear system of degree $b$ and the base locus $S \subset \BP^n$ of $\Phi$  is irreducible and nonsingular.
  In this paper,
we classify special birational transformations of type $(2,1)$. In addition to previous works \cite{AS2} and \cite{R2} on this topic, our proof employs natural $\cit^*$-actions on $Z$ in a crucial way.
These $\cit^*$-actions also  relate our result to the problem studied in \cite{FH}. \end{abstract}

\section{Introduction}\label{s.intro}

Recall (e.g. Section 2 in \cite{AS2} or Definition 4.1 in \cite{R2}) that a birational transformation
 $$\Phi: \BP^n \dasharrow Z \subset \BP^N$$ where $Z \subset \BP^N$ is a nonsingular projective variety of Picard number 1 is called  a {\em special birational transformation of type $(a,b)$} if \begin{itemize}
 \item[(1)] the base locus $S \subset \BP^n$ of $\Phi$
 is irreducible and nonsingular; \item[(2)] the rational map $\Phi$ is given by a linear system belonging to $\sO_{\BP^n}(a)$; and
 \item[(3)] the inverse rational map $\Phi^{-1}$ is given by a linear system belonging to $\sO_Z(b)$. \end{itemize}
When $Z = \BP^n$, this is a special Cremona transformation,  a classical topic in projective algebraic geometry.
It is a challenging problem to classify special birational transformations.
Even for special Cremona transformations, a complete classification is still missing.
Special Cremona transformations of type $(2,2)$ have been classified  by
Ein and Shepherd-Barron in \cite{ES} by
relating them to Severi varieties classified by Zak (\cite{Zak}).
 In \cite{R2}, special Cremona transformations of types  $(2, 3)$ and $(2, 5)$ have been classified.
 Recently Alzati and Sierra (\cite{AS2}) have extended \cite{ES} to  a classification of special birational transformations of type $(2, 2)$ for a wider class of $Z$.

In this paper,
we will give a complete classification of special birational transformations of type $(2,1)$. This classification can be described in terms of the classification of
the base locus $S \subset \BP^n$, which  is contained in a hyperplane $\BP^{n-1} \subset \BP^n$ for the type $(2,1)$.  Our main result is the following.

\begin{theorem} \label{t.intro}
The base locus
 $S^d \subset \BP^{n-1}$ of a special birational transformation of type $(2,1)$ is projectively equivalent to one of the following:
\begin{itemize}
\item[(a)] $\Q^d \subset \BP^{d+1}$ for $d \geq 1$;
\item[(b)]  $\pit^1 \times \pit^{d-1}  \subset \BP^{2d-1}$ for $d \geq 3$;
\item[(c)] the 6-dimensional Grassmannian ${\rm Gr}(2,5) \subset \pit^9$;
\item[(d)] the 10-dimensional Spinor variety $\mathbb{S}_5 \subset \BP^{15}$;
\item[(e)] a nonsingular codimension $\leq 2$ linear section of  $\pit^1 \times \pit^2 \subset \BP^5$;
\item[(f)]  a nonsingular codimension $\leq 3$ linear section of ${\rm Gr}(2,5) \subset \pit^9.$
\end{itemize}
\end{theorem}

The description of the varieties in Theorem \ref{t.intro} as well as a more precise formulation of the result will be given in Section \ref{s.statement}.
 Our proof of Theorem \ref{t.intro}, to be given in Section \ref{s.proof},  uses previous works on this topic
in \cite{AS2} and \cite{R2}. The main strategy is an inductive argument on VMRT (see Proposition \ref{p.pic}) developed by Russo in \cite{R2}.
 There are two  new ingredients in our approach:   the use of natural $\cit^*$-actions on $Z$, which reveals topological relations between the base loci of $\Phi$ and  $\Phi^{-1}$, and  a study of the intersection of entry loci on the base locus of $\Phi$, which exhibits a delicate  structure in the projective geometry of the base locus.    The part on $\cit^*$-action is
presented in Section \ref{s.Y} and the part on the intersection of entry loci is presented in Section \ref{s.entry}.
The use of $\cit^*$-actions on $Z$ is motivated by our previous work \cite{FH} on projective manifolds with nonzero
prolongations. As a matter of fact, we will see in Section \ref{s.prolongation} that a prime Fano manifold $Z$ is the target of a special birational transformation of type $(2,1)$ if and only if it has nonzero prolongation. By this correspondence, we can use
  Theorem \ref{t.intro} to give a new proof (see Theorem \ref{t.Main}) of the main classification result
 of \cite{FH}. This new proof  corrects an error in the classification in
 \cite{FH}, as explained in Remark \ref{r.FH}.

\bigskip
We will work over complex numbers. For simplicity, a nonsingular irreducible variety will be called a manifold.

\bigskip
{\em Acknowledgments:}  We are grateful to Jos\'e Carlos Sierra for explaining to us the examples in Proposition
\ref{p.ex1} and Proposition \ref{p.ex2} and for a careful reading of the draft.  We would like to thank Francesco Russo
for pointing out a mistake in a previous version and for numerous suggestions.
We are grateful to Qifeng Li for discussions on Section \ref{s.entry}.

\section{Statement of the classification of special quadratic manifolds}\label{s.statement}

\begin{definition}\label{d.quadrics}
Let $U$ be a vector space and let $\sigma: \SYM^2 U \to W$ be a surjective linear map to a vector space $W$.
 Let us denote by $W^* \subset \SYM^2 U^*$ the annihilator of ${\rm Ker}(\sigma)$, which is naturally dual to $W$.
  \begin{itemize} \item[(1)] Let $\psi^o: \BP U \dasharrow \BP W$ be the rational map defined by the linear
system $W^* \subset {\rm Sym}^2 U^* \simeq H^0(\BP U, \sO(2))$.
 The scheme-theoretic base locus of $\psi^o$ will be denoted by $B(\sigma)
\subset \BP U$ and the proper image of $\BP U$ under $\psi^o$ will be denoted by $Y(\sigma) \subset \BP W$.
\item[(2)]
Fix a 1-dimensional vector space $T$ with a fixed identification $T = \cit$.
Define a rational map $\phi^o: \BP (T \oplus U) \dasharrow \BP (T \oplus U \oplus W)$
by $$ [t: u] \mapsto  [t^2: tu: \sigma(u,u)] \mbox{ for } t \in T, u \in U.$$ The proper image of $\BP (T \oplus U)$ under $\phi^o$ will be denoted by
$Z(\sigma) \subset \BP (T \oplus U \oplus W)$.
The scheme-theoretic base locus of $\phi^o$ coincides with $B(\sigma) \subset \BP U = \BP (0 \oplus U)$.
\end{itemize}
We will say that $\sigma$ is a {\em special system of quadrics} if (a)  the base locus subscheme $B(\sigma) \subset \BP U$ is  irreducible and nonsingular and (b) the image $Z(\sigma)$ of $\phi^o$ is nonsingular.
A projective manifold $S \subset \BP U$ is called a {\em special quadratic manifold} if $S = B(\sigma)$ for a special system of quadrics $\sigma: \SYM^2 U \to W$.

In this case, the rational map $\psi^o$ comes from a morphism $$\psi: \Bl_S(\BP U) \to Y(\sigma) \subset \BP W$$ where $\Bl_S(\BP U)$ is the
blow-up of $\BP U$ along $S$. We will denote by $F\subset \Bl_S(\BP U)$ its exceptional divisor.
The rational map $\phi^o$ comes from a morphism $$\phi: \Bl_S(\BP (T \oplus U)) \to Z(\sigma) \subset \BP (T \oplus U \oplus W)$$ where $\Bl_S(\BP(T \oplus U))$ is the
blow-up of $\BP (T\oplus U)$ along $S \subset \BP U= \BP (0 \oplus U)$. We will denote by $E\subset \Bl_S(\BP (T \oplus U))$ its exceptional divisor.
We have the
 commuting diagram $$ \begin{array}{ccccccc} F & \subset & \Bl_S(\BP U ) & \stackrel{\psi}{\longrightarrow} & Y(\sigma) & \subset & \BP W \\
 \cap & & \cap & & \cap & & \cap \\
 E & \subset & \Bl_S(\BP (T \oplus U)) & \stackrel{\phi}{\longrightarrow} & Z(\sigma) & \subset & \BP (T \oplus U \oplus W). \end{array} $$
 Note that $\phi$ sends $U \cong \BP (T \oplus U) \setminus \BP U$ isomorphically to $Z(\sigma) \setminus \BP (U \oplus W).$ Thus $Z(\sigma)$ is a rational variety and $\phi^o: \BP (T \oplus U) \dasharrow Z(\sigma)$ is a birational map.
 \end{definition}

It follows  immediately  from the definition that for a
special quadratic manifold $S \subset \pit U$,  the birational map $\phi^o: \pit(T \oplus U) \dasharrow Z(\sigma)$ is a special birational transformation of type $(2,1)$ in the sense of Section \ref{s.intro}. Conversely we have
\begin{proposition}
Let $\Phi: \pit^n \dasharrow Z \subset \pit^N$ be a  special birational transformation of type $(2,1)$ as defined in Section \ref{s.intro}.
Then the base locus $S$ of $\Phi$ is contained in a hyperplane $\pit^{n-1}$ of $\pit^n$, the subvariety
$S \subset \pit^{n-1}$ is a special quadratic manifold and $\Phi$ coincides with the birational map
$\phi^o$ for a special system of quadrics $\sigma: \SYM^2 U \to W$.
\end{proposition}
\begin{proof}
By Proposition 2.3 (a) \cite{ES}, the base locus $S \subset \pit^n$ of $\Phi$ is contained in a hyperplane. Since $S$ is defined by quadratic equations,  $S \subset \pit^{n-1}$ is also defined by quadratic equations.
As $S$ is the base locus of $\Phi$,  the rational map $\Phi$ is given by a linear subspace $W^* \subset H^0(\pit^{n-1}, \mathcal{I}_S(2))$ and the full linear system $H^0(\pit^n, \mathcal{I}_{\pit^{n-1}}(2))$. Then the map
$\phi^o$ associated to $W$ coincides with $\Phi$.
\end{proof}

\begin{example}\label{e.table}
We list some homogeneous examples of  special quadratic manifolds in the next table.  The data in this table can be
 found from Theorem 3.8 in Chapter III of \cite{Zak}. Note that in all these examples, the dimension $a$ of $W$ is  equal to $ h^0(\pit^{n-1}, \mathcal{I}_S(2)),$ namely, we have $W^* = H^0(\pit^{n-1}, \mathcal{I}_S(2))$.
\end{example}

\medskip
\begin{center}

\begin{tabular}{|c||c|c|c|c|}
\hline $S$ & $\qit^{d}$ & $\pit^1 \times \pit^{d-1}$ & $\Gr(2, 5)$ & $\mathbb{S}_5$ \\
\hline $S \subset \pit^{n-1}$ & $\sO(1)$ & Segre & Pl\"ucker & Spinor \\
\hline $Y$ & point & $\Gr(2, d)$ & $\pit^4$ & $\qit^8$ \\
\hline $Y \subset \BP^{a-1}$ & identity & Pl\"ucker & identity & $\sO(1)$ \\
\hline $Z$ & $\qit^{d+2}$ & $\Gr(2, d+2)$ & $\mathbb{S}_5$ &  $\mathbb{OP}^2$ \\
\hline $Z\subset \BP^{n+a}$ & $\sO(1)$ & Pl\"ucker & Spinor & Severi \\
\hline $d= \dim S$ & $d = n-2$ & $d$ & 6 & 10 \\
\hline $n = \dim Z$ & $n= d+2$ & $2d$ & 10 & 16 \\
\hline $m= \dim Y$ & $0$ & $2(d-2)$ & 4 & 8 \\
\hline $a$ & 1 & $d(d-1)$ & 5 & 10 \\ \hline
\end{tabular} \end{center}

\begin{example} \label{e.projection}
Here we give some examples of special birational transformations of type $(2, 1)$ with $W^* \subsetneq H^0(\pit^{n-1}, \mathcal{I}_S(2))$. Recall from Example \ref{e.table} that the special quadratic manifold $S = \pit^1 \times \pit^{d-1} \subset \pit^{2d-1}$ with $d \geq 6$ is associated to the special system of quadrics
$$
\sigma: {\rm Sym}^2 U \to W = H^0(\pit U, \mathcal{I}_S(2))^*.
$$
The corresponding birational transformation is $$\phi^\circ: \pit (T \oplus U) \dasharrow Z:={\rm Gr}(2, d+2) \subset \pit(T \oplus U \oplus W)$$ with $Y = {\rm Gr}(2, d) \subset \pit W = \pit (\wedge^2 \cit^d)$.
Note that ${\rm Sec}(Z) \cap \pit W = {\rm Sec}(Y)$. Take any linear subspace
$L \subset W$ such that $\pit L \cap {\rm Sec}(Y) = \emptyset$, then $\pit L \cap {\rm Sec}(Z) =\emptyset$.
Let $$p_L: \pit(T \oplus U \oplus W) \dasharrow \pit(T \oplus U \oplus W/L)$$ be the projection from $\pit L$. Then $p_L$ sends  $Z$ (resp. $Y$) isomorphically to a subvariety $$Z_L \subset  \pit(T \oplus U \oplus W/L)  \mbox{ (resp.   } Y_L \subset \pit(W/L)). $$ The map $$\phi_L^\circ:=p_L \circ \phi^\circ: \pit(T \oplus U) \dasharrow Z_L$$ is a special birational transformation of type $(2, 1)$ associated to the special system of quadrics
$$
\sigma_L: {\rm Sym}^2 U \to W/L.
$$
with $Y(\sigma_L) = Y_L \subset \pit(W/L)$ and  $B(\sigma_L) \subset \pit U$ being $\pit^1 \times \pit^{d-1} \subset \pit^{2d-1}$.
\end{example}

To discuss non-homogeneous examples, it is convenient to introduce  the following notion.

\begin{definition}\label{d.dual}
Let $Z \subset \BP V$ be a nondegenerate submanifold and let $W \subset V$ be a subspace such that $\BP W \subset Z$.
Denote by $(V/W)^* \subset V^*$  the set of linear functionals on $V$ annihilating $W$ such that $\BP (V/W)^*$
parameterizes the set of hyperplanes in $\BP V$ containing $\BP W$. Then a general member of $\BP (V/W)^*$ is
called a $\BP W$-{\em general} hyperplane in $\BP V$. More generally, a linear subspace of codimension-$s$ in $\BP V$ is
$\BP W$-{\em general} if it is defined by a general member of $\Gr(s, (V/W)^*)$, i.e., it is general among
subspaces of codimension-$s$ containing $\BP W$.  \end{definition}

 \begin{proposition}\label{p.hyperplane}
 Let $S = B(\sigma) \subset \pit U$ be a special quadratic manifold defined by $\sigma: \SYM^2 U \to W$.
 Assume that  $\dim S \geq 2$,  $\dim \BP U > \dim Y(\sigma)$ and  the intersection of $Z(\sigma)$ with a $\BP W$-general hyperplane
of $\BP (T \oplus U \oplus W)$ is nonsingular. Then for  a general subspace $U' \subset U$ of  codimension 1,
 the restriction $\sigma': \SYM^2 U' \to W$ of $\sigma$ is  a special system of quadrics such that
 \begin{itemize} \item[(i)] the base locus scheme $B(\sigma')
 \subset \BP U'$ coincides with the hyperplane section $S' = S \cap \BP U' \subset \BP U$ of $S = B(\sigma)$;
 \item[(ii)] $Y(\sigma') = Y(\sigma) \subset \BP W$; and
   \item[(iii)]  $Z(\sigma') = Z(\sigma) \cap \BP (T \oplus U' \oplus W) .$  \end{itemize} \end{proposition}

 \begin{proof}
Firstly, we claim that for a general $U' \subset U$, the hyperplane section $Z(\sigma) \cap \BP (T \oplus U' \oplus W)$ is nonsingular. This is a consequence of the assumption  that a $\BP W$-general hyperplane section of $Z(\sigma) \subset \BP (T \oplus U \oplus W)$ is nonsingular. To see this,  associate to each vector $v \in U$ the linear automorphism $g_v$ of $\BP (T \oplus U \oplus W)$ defined by
 $$g_v: [t: u: w] \mapsto [t: u + tv: w + 2 \sigma(u,v) +  t\sigma(v, v)].$$  For a general choice of $v \in U$ and $U' \subset U$,
 the automorphism $g_v$ sends $\BP (T \oplus U' \oplus W)$ to a $\BP W$-general hyperplane of $\BP (T \oplus U \oplus W)$.
The $g_v$-image of  a general point $[1: u : \sigma(u,u)] \in Z(\sigma)$ is
\begin{eqnarray*} g_v([1:u: \sigma(u,u)]) & = & [1: u+v: \sigma(u,u) + 2\sigma(u,v) + \sigma(v,v)] \\ &=& [1: u+v: \sigma(u+v, u+v)] \in Z(\sigma).\end{eqnarray*}
Thus $g_v$ preserves $Z(\sigma)$. Consequently, for a general choice of $v$, the automorphism $g_v$ sends the hyperplane section $Z(\sigma) \cap \BP (T \oplus U' \oplus W)$ to a $\BP W$-general hyperplane section of $Z(\sigma)$. This proves the claim.

 Now we can choose a general subspace  $U' \subset U$ of codimension 1 such that
 \begin{itemize} \item[(1)]   $\dim \BP U' \geq \dim Y(\sigma)$ and $\dim U - \dim U' < \dim S$;
 \item[(2)] the (scheme-theoretic) linear section $S':=S \cap \BP U'$  is  nonsingular; and
 \item[(3)] the restriction $\psi^o|_{\BP U'}$ is dominant over $Y(\sigma)$. \end{itemize}
 Then $\sigma' : \SYM^2 U' \to W$ is surjective, and $S' = B(\sigma') \subset \pit U'$ is a special quadratic manifold with $Y(\sigma') = Y(\sigma)$, and
 $Z(\sigma') = Z(\sigma) \cap \BP (T \oplus U' \oplus W)$.  \end{proof}

To check the condition in Proposition \ref{p.hyperplane}, we need the following.

\begin{proposition}\label{p.dual}
In the setting of Definition \ref{d.dual}, let $Z \subset \BP V$ be a nondegenerate submanifold containing
$\BP W$.
Define $Z^*_W \subset \BP V^*$ by
$$Z^*_W = \{ [H] \in \BP (V/W)^* \ \mid \ H \cap Z \mbox{ is singular at a point of }
\BP W \}.$$ Assume that $\dim Z^*_W < \dim \BP (V/W)^*.$ Then for a $\BP W$-general hyperplane $[H] \in \BP (V/W)^*$,
the intersection $\sZ := Z \cap H$ is a nonsingular subvariety containing $\BP W$ and the submanifold $\sZ \subset H$
satisfies $\dim \sZ^*_W \leq \dim Z^*_W.$ \end{proposition}

\begin{proof}
The intersection of $Z$ with a $\BP W$-general $H$  is nonsingular
outside $\BP W$ by Bertini. But we can choose $[H] \in \BP (V/W)^*$ outside  $Z^*_W$ by
the dimension condition $\dim Z^*_W < \dim \BP (V/W)^*.$ Thus $\sZ = Z \cap H$ is nonsingular and it contains $\BP W$.
It remains to check  $\dim \sZ^*_W \leq \dim Z^*_W.$

By definition, $$\sZ^*_W = \{ [L] \in \BP (H/W)^* \ \mid \ L \cap \sZ \mbox{ is singular at a point of }
\BP W\}.$$ Given an element $[L]$ of $\sZ^*_W$,    then $L \cap \sZ$ is singular at some point say $x \in \BP W$.
Take the hyperplane $\widetilde{L} \subset V $ to be the linear span of $L$ and $T_xZ$, then
$L = \widetilde{L} \cap H$ and $\widetilde{L} \cap Z$ is singular at $x \in \BP W$. This shows that the image of $Z_W^*$ under the projection
$\BP(V/W)^* \setminus \{H\} \to \BP (\hat{H}/W)^*$  contains $\sZ_W^*$, where
$\hat{H}$ is the hyperplane in $V$ corresponding to $H$.
This implies that
 $\dim \sZ^*_W \leq \dim Z^*_W.$ \end{proof}

By applying Proposition \ref{p.dual} repeatedly,
we have the following.

\begin{corollary}\label{c.dual}
In Proposition \ref{p.dual}, let $s$ be a positive integer satisfying $$s < \dim Z \mbox{ and } \dim Z^*_W \leq \dim \BP (V/W)^* -s.$$
Then a $\BP W$-general linear section of $Z$ with codimension $s$ is nonsingular. \end{corollary}

Let us recall the following results from Proposition 2.19 and Remark 2.20 in Chapter III of \cite{Zak}.

\begin{proposition}\label{p.Zak}
(i) Let $Y=\BP W \subset Z \subset \BP V$ be  $\BP^2 \subset \Gr(2,5) \subset \BP^9$ from $S = \BP^1 \times \BP^2$ of Example \ref{e.table}.
Then $Z^*_W$ is isomorphic to a cone over $S$. In particular, $\dim \BP (V/W)^*-\dim Z^*_W = 6-4 =2.$

(ii) Let $Y=\BP W \subset Z \subset \BP V$ be  $\BP^4 \subset \mathbb{S}_5 \subset \BP^{15}$ from $S = \Gr(2,5)$ of Example \ref{e.table}.
Then $Z^*_W$ is isomorphic to a cone over $S$. In particular, $\dim \BP (V/W)^*-\dim Z^*_W = 10-7 =3.$
\end{proposition}

\begin{lemma} \label{l.equations}
Recall that a subvariety $S \subset \pit^{n-1}$ is  {\em arithmetically Cohen-Macaulay} if $H^i(\pit^{n-1}, \mathcal{I}_S(l))=0$ for all $l$ and for all $0 < i < \dim(S)+1$.
Let $S \subset \pit^{n-1}$ be either the Segre embedding $\pit^1 \times \pit^2 \subset \pit^5$ or
the Pl\"ucker embedding ${\rm Gr}(2,5) \subset \pit^9$. Let $S' \subset \pit^{n-s-1}$ be any nonsingular linear section of $S \subset \pit^{n-1}$ with  codimension $s < \dim S$.  Then $S' \subset \pit^{n-s-1}$  is arithmetically Cohen-Macaulay and  $H^0(\pit^{n-1},\mathcal{I}_S(2)) = H^0( \pit^{n-s-1}, \mathcal{I}_{S'}(2))$.
\end{lemma}
\begin{proof}
By \cite{Zak} (Chapter III, Theorem 1.2),  $S\subset \pit^{n-1}$  is arithmetically Cohen-Macaulay. On the other hand, for a nonsingular hyperplane section, we have the exact sequence
$$
0 \to \mathcal{I}_S(-1) \rightarrow \mathcal{I}_S \to \mathcal{I}_{S\cap H \subset H} \to 0.
$$
Using the associated long exact sequence, we deduce that $S \cap H$ is arithmetically Cohen-Macaulay. Repeating the same argument,  we see that any nonsingular linear section  of $S$ is arithmetically Cohen-Macaulay. The last claim then  follows easily.
\end{proof}

Now we can give some nonhomogeneous examples of special quadratic manifolds.

\begin{proposition}\label{p.ex1}
From Example \ref{e.table}, the Segre variety $S=\pit^1 \times \pit^2 \subset \pit^5 = \BP U$ is a special quadratic manifold associated with
$$\sigma: \SYM^2 U \to W = H^0(\BP^5, \sI_S(2))^*$$ and $Z(\sigma)$ isomorphic to the Pl\"ucker embedding $\Gr(2,5) \subset \BP^{9}$.  A nonsingular linear section $S'$ of  $S$ by a general subspace $U' \subset U$ of codimension $s \leq 2$
is a special quadratic manifold associated to a special system of quadrics $$\sigma': \SYM^2 U' \to W = H^0(\BP^5, \sI_S(2))^*= H^0(\BP^{5-s}, \sI_{S'}(2))^*$$  and $Z(\sigma')$  is equal to a nonsingular linear section of $\Gr(2,5)$
with codimension $s$.
\end{proposition}

\begin{proof}
Applying Proposition \ref{p.hyperplane} repeatedly in combination with Corollary \ref{c.dual} and Proposition \ref{p.Zak}, we see  that a general linear section $S'$
of $\pit^1 \times \pit^2$ with codimension $s \leq 2$ is a special quadratic manifold associated to a special system of
 quadrics $\sigma': \SYM^2 U' \to W = H^0(\BP^5, \sI_S(2))^*$ and $Z(\sigma')$  equal to a
 $\BP^2$-general linear section of $\Gr(2,5)$
with codimension $s$.  It is well-known (e.g. Remark 3.3.2 in \cite{IP})
 that  all nonsingular sections of $S$ with codimension $s\leq 2$ are projectively equivalent and
all nonsingular sections of $\Gr(2,5)$ with codimension $s \leq 2$ are projectively equivalent.
Thus we may say that $Z(\sigma')$ is any nonsingular linear section of $\Gr(2,5)$ of codimension $s \leq 2$.
Using Lemma \ref{l.equations}, we see $W^*= H^0(\BP^{5-s}, \sI_{S'}(2))$.
\end{proof}

\begin{proposition}\label{p.ex2}
From Example \ref{e.table}, the Grassmannian $S=\Gr(2,5)$ $\subset$ $\pit^{9} = \BP U$ is a special quadratic manifold associated with
$$\sigma: \SYM^2 U \to W = H^0(\BP^{9}, \sI_S(2))^*$$ and $Z(\sigma)$ is isomorphic to the Spinor variety $\mathbb{S}_5 \subset \BP^{15}$.  A nonsingular linear section $S'$ of  $S$ by a general subspace $U' \subset U$ of codimension $s \leq 3$
is a special quadratic manifold associated to a special system of quadrics $$\sigma': \SYM^2 U' \to W = H^0(\BP^{9}, \sI_S(2))^*= H^0(\BP^{9-s}, \sI_{S'}(2))^*$$  and $Z(\sigma')$  equals to a $\BP^4$-general linear section of $Z(\sigma) = \mathbb{S}_5$
with codimension $s \leq 3$.
\end{proposition}

\begin{proof}
As in the proof of Proposition \ref{p.ex1}, this follows from Proposition \ref{p.hyperplane}, Corollary \ref{c.dual}, Proposition \ref{p.Zak} and Lemma \ref{l.equations}, modulo the fact
(see Remark 3.3.2 in \cite{IP}) that
all nonsingular linear sections of $\Gr(2,5)$ of a fixed codimension $s \leq 3$ are projectively equivalent.
\end{proof}

\begin{remark} (1) Let $\mathbb{S}_5 = S \subset \pit U, \dim U = 16,$ be the 10-dimensional Spinor variety. It is well-known that
 the dual variety $S^* \subset \pit U^*$ is isomorphic to $S \subset \pit U$ and the automorphism group ${\rm Aut}(S)$ acts transitively on $S$ and $\pit U \setminus S.$ This implies  that  all  nonsingular hyperplane sections of $\mathbb{S}_5$ are projectively equivalent.

 (2) A general  linear section of $\mathbb{S}_5$ with codimension 2 does {\em not} contain a linear
 $\pit^4$.
 One way to see this is using the fact that a general hyperplane section $S_1$ of $\mathbb{S}_5 \subset \pit^{15}$ is isomorphic to a horospherical Fano manifold of Picard number 1, the case 4 in Theorem 1.7 of \cite{Pasquier} (this fact follows from Mukai's classification \cite{M}). From \cite{Pasquier}, the automorphism group
${\rm Aut}(S_1)$ has two orbits, an open orbit and a closed orbit,
say $Q \subset S_1$,  which is isomorphic to the 6-dimensional
hyperquadric $\qit^6$.  Let $\pi: {\rm Bl}_Q(S_1) \to S_1$ be the
blow-up of $S_1$ along $Q$ and let $E \subset {\rm Bl}_Q(S_1)$ be
the exceptional divisor. Then by the proof of Lemma 1.17
\cite{Pasquier}, there exists a morphism $q: {\rm Bl}_Q(S_1) \to
\qit^5$ which is a $\pit^4$-bundle and $E$, which is a
$\BP^3$-bundle over $\qit^5$ under the map  $q$, is homogeneous.
The fibers of $q$ are mapped to linear $\pit^4$'s contained in
$S_1$ and any linear $\pit^4$ in $S_1$ arises this way. This
implies that the intersection of $\pit^4$ in $S_1$ with $Q$ is
$\pit^3 \subset \pit^4$. Now take a general hyperplane section $H
\subset S_1$, then $H \cap Q$ is a smooth $\qit^5$. If $H$
contains a $\pit^4$, then  $H \cap Q$ will contain a  $\pit^3$, a
contradiction.

\end{remark}

The following result from  Corollary 3.21 in \cite{AS2} is a  converse to Proposition \ref{p.ex1} and Proposition \ref{p.ex2}.

\begin{proposition}\label{p.AS2}
A nonsingular linear section of the homogeneous special quadratic manifolds in
Example \ref{e.table} is a special quadratic manifold if and only if  it is a hyperquadric or  one  provided by Proposition \ref{p.ex1} and Proposition \ref{p.ex2}.  \end{proposition}

Our aim is to show that Example \ref{e.table}, Example \ref{e.projection},  Proposition  \ref{p.ex1} and Proposition  \ref{p.ex2}
exhaust all special quadratic manifolds.
More precisely, we have the following classification,  which gives a complete classification of special birational transformations of type (2,1).

\begin{theorem} \label{t.bir}
A special quadratic manifold
 $S^d \subset \BP^{n-1}$ is projectively equivalent to one of the following:
\begin{itemize}
\item[(a)] $\Q^d \subset \BP^{d+1}$ for $d \geq 1$;
\item[(b)]  $\pit^1 \times \pit^{d-1}  \subset \BP^{2d-1}$ for $d \geq 3$;
\item[(c)] the 6-dimensional Grassmannian ${\rm Gr}(2,5) \subset \pit^9$;
\item[(d)] the 10-dimensional Spinor variety $\mathbb{S}_5 \subset \BP^{15}$;
\item[(e)] a nonsingular codimension $\leq 2$ linear section of  $\pit^1 \times \pit^2 \subset \BP^5$;
\item[(f)]  a nonsingular codimension $\leq 3$ linear section of ${\rm Gr}(2,5) \subset \pit^9.$
\end{itemize}

 The birational map $\psi^\circ$ is given by $H^0(\pit^{n-1}, \mathcal{I}_S(2))$ except for $(b)$
with $d \geq 6$, where $\psi^\circ$ can be given by some proper subspaces of $H^0(\pit^{n-1}, \mathcal{I}_S(2))$ as in Example \ref{e.projection}.
\end{theorem}
The proof of Theorem \ref{t.bir} will be given in Section \ref{s.proof}.

\section{Classification of $Y(\sigma)$}\label{s.Y}

In this section, we classify varieties that
can occur as the $\psi$-image $Y(\sigma)$ of a special quadratic manifold $S \subset \BP U$.
 To simplify the notation, we will write $Y$ for $Y(\sigma)$ and $Z$ for $Z(\sigma)$.  We will use the notation of Definition \ref{d.quadrics}.

\begin{proposition}\label{p.Ysmooth}
Let $S \subset \BP U$ be a special quadratic manifold. Then $Y$ is nonsingular.\end{proposition}

\begin{proof}
Consider the $\cit^*$-action on $\pit (T \oplus U)$ given by $\lambda \cdot [t:u] = [t: \lambda u]$ for all $\lambda \in \cit^*$.
The fixed locus of this $\cit^*$-action has two components: the isolated point $ \BP T$ and  the hyperplane $\pit U.$
  As $S$ is contained in the fixed locus, this $\cit^*$-action lifts to $\Bl_S(\pit (T \oplus U))$. Since the morphism $\phi: \Bl_S(\pit (T \oplus U)) \to Z$ has connected fibers by Zariski main theorem,  the $\cit^*$-action on $\Bl_S(\pit(T \oplus U))$ descends to a $\cit^*$-action
on $Z$ such that $\phi$ is $\cit^*$-equivariant. This implies that the proper image of $\pit U$ under $\phi^o$, which is $Y$, is contained in
the fixed locus of the $\cit^*$-action on $Z$.  On the other hand,  the $\cit^*$-orbit of a general point of
$Z$ has a limit point in $Y$, hence $Y$ is an irreducible component of the fixed locus of the $\cit^*$-action
on $Z$. This shows that $Y$ is nonsingular because $Z$ is nonsingular. \end{proof}

\begin{definition}\label{d.secant}
For a projective submanifold $M \subset \BP^N$ and two distinct points $x \neq y \in M$, denote by $\ell_{x,y} \subset \BP^N$
the line joining $x$ and $y$. Such a line is called a {\em secant line} of $M$.   The {\em secant variety} of $M$ is $${\rm Sec}(M) = \mbox{ the closure of } \cup_{x\neq y \in M} \ell_{x,y}.$$
  The {\em secant defect} of $M$ is defined by  $\delta_M= 2 \dim M + 1 -\dim {\rm Sec}(M).$
\end{definition}

\begin{proposition}\label{p.secant}
Let $S \subset \BP U$ be a special quadratic manifold.  Let  $\pi: \Bl_S(\BP U) \to \BP U$ be  the blow-up morphism. Then ${\rm Sec}(S) = \BP U$ and
the morphism $\psi: \Bl_S(\BP U) \to Y$ has the following properties:
\begin{itemize}
\item[(1)] $\psi$ contracts the proper transform $\pi^{-1}[\ell]$ of a secant line $\ell$  of $S$ not contained in $S$  to a point in $Y$;
\item[(2)]  the fibers of the exceptional divisor  $\pi|_F: F\to S$ are sent isomorphically to linear subspaces in $Y$; and
\item[(3)] $\psi (F) = Y$.
\end{itemize}\end{proposition}

\begin{proof}
The equality ${\rm Sec}(S) = \BP U$ is from Proposition 2.3(a) of \cite{ES}.

Recall that $\psi^o: \BP U \dasharrow Y $ is induced by the projection $\sigma: \SYM^2 U \to W$ composed with
 the second Veronese embedding $v_2: \BP U \to \BP \SYM^2 U$.
Let $\ell \subset \BP U$ be a secant line of $S$ not contained in $S$.  Let $L \subset \SYM^2 U$ be the 3-dimensional subspace
spanned by $v_2(\ell).$ Since $v_2(\ell)$ intersects $\BP {\rm Ker}(\sigma)$ at two distinct points,
we have $\dim L \cap {\rm Ker}(\sigma) = 2$. Thus $\sigma(L)$ is 1-dimensional, i.e., $\psi$ contracts $\ell$, proving (1).

Note that $\psi$ can be viewed as the restriction of the natural projection $p: \Bl_{\BP {\rm Ker}(\sigma)}(\BP \SYM^2 U) \to \BP W$
to $\Bl_S(v_2(\BP U))$. As the normal bundle of $\BP {\rm Ker}(\sigma)$ in $\BP \SYM^2U$ is $\0(1)^{s}$ for some $s$,
the exceptional divisor of the blow-up is a product of two projective spaces, which can only be contracted by the projection to the other factor.
Hence the  projection $p$ sends a fiber of the blow-up  isomorphically to a linear subspace in $\BP W$. This implies (2).

Finally, (3) is a consequence of ${\rm Sec}(S) = \BP U$ and (1). \end{proof}

Let us recall the following basic terminology.

 \begin{definition}\label{d.cc}
Let $M \subset \BP^N$ be a projective submanifold. \begin{itemize}
\item[(1)] $M$ is {\em conic-connected} if there exists an irreducible conic curve through two general points of $M$.
\item[(2)] $M$ is a {\em prime Fano manifold} if  ${\rm Pic}(M) $
is generated by $\0_M(1)$ and $M$ is covered by lines. \end{itemize} \end{definition}

 \begin{proposition}\label{p.Ypf}
 Let $S \subset \BP U$ be a special quadratic manifold. Then $Y$ is a conic-connected prime Fano manifold.
 \end{proposition}

 \begin{proof}
Recall that $\psi^o: \BP U \dasharrow Y$ is induced by the projection $\sigma: \SYM^2 U \to W$ composed with
 the second Veronese embedding $v_2: \BP U \to \BP \SYM^2 U$.
 Since the Veronese variety $v_2(\BP U) \subset \BP (\SYM^2 U)$
 is conic-connected, the image $Y$ must be conic-connected.

Since $\Bl_S(\BP U)$ has Picard number 2 and $\psi: \Bl_S(\BP U) \to Y$ contracts some curves by Proposition \ref{p.secant}
(1), we see that
$Y$ has Picard number 1. By Proposition \ref{p.secant} (2), lines cover $Y$. Thus $Y$ is a prime Fano manifold. \end{proof}

 \begin{proposition}\label{p.number}
  The secant defect $\delta = \delta_S$ of a special quadratic manifold $S \subset \pit U$ satisfies
 $\delta = 2 \dim S +2 - \dim U$
and  $$\dim Y = 2(\dim S - \delta) = 2 (\dim U - \dim S-2).$$
 \end{proposition}

 \begin{proof} The equality
 $\delta = 2 \dim S +2 - \dim U$ is a consequence of ${\rm Sec}(S) = \BP U$ in Proposition \ref{p.secant} and the equality
 $\dim Y = 2(\dim S - \delta) $ is from Corollary 2.3 of \cite{AS2}.  Finally, $\dim S - \delta = \dim U - \dim S -2$
 from $\delta = 2 \dim S + 2 - \dim U$. \end{proof}

 \begin{proposition}\label{p.blow}
 Let $S\subset \BP U$ be a special quadratic manifold with secant defect $\delta$. Then there is a natural identification
 $\Bl_S(\BP (T \oplus U) )= \Bl_Y(Z)$
 such that $\phi: \Bl_S(\BP (T \oplus U)) \to Z$ coincides with the blow-up of $Z$ along $Y$ with the exceptional divisor $\Bl_S(\BP U)$. In particular, the morphism $\psi: \Bl_S(\BP U) \to Y$ is a $\pit^{\delta +1}$-bundle. \end{proposition}

\begin{proof}
Since both $Y$ and $Z$ are nonsingular,
   Proposition \ref{p.secant} and Theorem 1.1 of \cite{ES}  imply that
$\Bl_S(\pit(T \oplus U))$ is also the blow-up of $Z$ along $Y$.  Thus $\psi$ is a $\pit^k$-bundle, where
$k = \dim Z - \dim Y-1$. By Proposition \ref{p.number}, we have $k = \delta +1$.   \end{proof}

 The following two propositions are direct consequences of Proposition \ref{p.blow}.

\begin{proposition}\label{p.linear}
 Let $S\subset \BP U$ be a special quadratic manifold with secant defect $\delta$. A general fiber of the $\pit^{\delta +1}$-bundle
 $\psi: \Bl_S(\BP U) \to Y$ in Proposition \ref{p.blow} is sent to a linear subspace in $\BP U$ by the blow-up morphism $\Bl_S(\BP U) \to \BP U$.
 \end{proposition}

 \begin{proof} Let $z \in {\rm Sec}(S) = \BP U$ be a general point.
Let $C_z(S) \subset \BP U$ be the union of secant lines of $S$ passing through $z$.
By Proposition 2.3 (b) of \cite{ES}, the cone $C_z(S) \subset \BP U$ is
 a linear subspace of dimension $\delta +1$. But by Proposition \ref{p.secant} (1), the cone $C_z(S)$ is contained in the image of a fiber of
 $\psi$. Thus $C_z(S)$ must coincide with the image of a fiber of $\psi$. \end{proof}

\begin{proposition}\label{p.Euler}
For a special quadratic manifold $S$, set $n= \dim U$ and $d = \dim S$. Then
the Euler numbers of $S$ and $Y$ are related by
$$ \chi(S) = \frac{(\delta+2) \chi(Y) -n}{n-d-2}.$$
\end{proposition}

\begin{proof}
 Since the exceptional divisor  $F$ of the  blow-up ${\rm Bl}_S(\pit^{n-1}) \to \pit^{n-1}$ is a $\BP^{n-d-2}$-bundle over $S$,  we have
$$ \chi(\pit^{n-1}) - \chi(S) = \chi({\rm Bl}_S(\pit^{n-1})) - \chi(F) =\chi({\rm Bl}_S(\pit^{n-1})) -  (n-d-1) \chi(S),$$
which gives $\chi({\rm Bl}_S(\pit^{n-1})) = n + (n-d-2) \chi(S)$.
On the other hand, the map ${\rm Bl}_S(\pit^{n-1}) \to Y$ is a $\pit^{\delta+1}$-bundle from Proposition \ref{p.blow}, which implies  $\chi({\rm Bl}_S(\pit^{n-1})) = (\delta+2) \chi(Y)$. Combining the two gives the desired equality. \end{proof}

We are ready to have a classification of $Y$.

\begin{theorem}\label{t.Sato}
Let $S \subset \BP U$ be a special quadratic manifold and let $c =
\dim U - \dim S -1$ be its codimension. Recall that $\dim Y =
2(c-1)$ from
  Proposition \ref{p.number}. Then   $Y \subset \BP W$ is isomorphic to one of the following: \begin{itemize}
\item[(Y1)] $ \BP^{2(c-1)} \cong Y = \BP W$;
\item[(Y2)]  a nonsingular quadric hypersurface $\Q^{2(c-1)} \cong Y \subset \BP W \cong \BP^{2c-1}$;  or
\item[(Y3)]  a biregular projection of the Pl\"ucker embedding ${\rm Gr}(2, \C^{c+1}) \subset \BP (\wedge^2 \C^{c+1}).$
\end{itemize}
 \end{theorem}

 \begin{proof} By Proposition \ref{p.secant}, the (nondegenerate) projective submanifold $Y \subset \BP W$ of dimension $2 (c-1)$ is covered by
 linear subspaces of dimension $c-1$. Thus
 we can use Eichi Sato's classification, Main Theorem  in \cite{Sato} (see also \cite{NO} Corollary 5.3, which completes the classification), of  nonsingular projective varieties of dimension $\leq 2(c-1)$ covered by linear subspaces of
 dimension $(c-1)$.
  Since $Y$ has Picard number 1 from Proposition \ref{p.Ypf}, Sato's classification shows that  $Y \subset \BP W$ must be one of the three  listed varieties. \end{proof}

We have the following topological consequence.

\begin{corollary} \label{c.Betti}
Let $S \subset \BP U$ be a special quadratic manifold with $d= \dim S$ and $n= \dim U$.  Then all odd Betti numbers of $S$ vanish. In particular,   the Euler number of $S$ satisfies $\chi(S) \geq d+1$.
\end{corollary}

\begin{proof}
 As the morphism ${\rm Bl}_S(\pit^{n-1}) \to Y$ is a $\pit^{\delta+1}$-bundle by Proposition \ref{p.blow} and odd Betti numbers of $Y$ vanish
from Theorem \ref{t.Sato},  the odd Betti numbers of  ${\rm Bl}_S(\pit^{n-1})$ vanish.
This implies that odd Betti numbers of $S$ vanish by the formulae for blow-ups in  Chapter 4, Section 6, p605 of \cite{GH}.
\end{proof}

In the case of (Y2) of Theorem \ref{t.Sato}, we have the following topological consequence.

\begin{proposition} \label{p.Q_delta}
For a special quadratic manifold $S^d \subset \pit^{n-1}$ with codimension $c$ and $n \geq 3$,  assume that $Y =\qit^{2(c-1)}$.
Then either $\delta \geq \frac{d}{2}$ or $\delta = \frac{d}{2}-1$.
\end{proposition}
\begin{proof}
Putting $\chi(\qit^{2(c-1)}) = 2c$ in Proposition \ref{p.Euler}, we obtain
\begin{equation}\label{e.euler}
\chi(S) = 2 \delta + \cfrac{n}{n-d-2}.
\end{equation}
Set $k:= \cfrac{n}{n-d-2}.$ The above equation implies that $k$ is a natural number, hence $2 \leq k\leq n$.

 Assume that $\delta < \frac{d}{2}$, i.e. $2d+2-n <\frac{d}{2}$, which is equivalent to
$3d+5 \leq 2n$,  implying $d \leq \cfrac{2n-5}{3}$ and
$$n = k(n-d-2) \geq k(n-\cfrac{2n-5}{3}-2) = \frac{k(n-1)}{3}.$$
Let us check case-by-case the possible values of $k$.
\begin{itemize} \item[(1)] When $4 \leq k \leq n$. Then $3n \geq 4(n-1)$ gives
 $n=4, d=1, \delta=0$ and $\chi(S)= 4$. But there is no smooth curve with $\chi(S)=4$, a contradiction.
\item[(2)] When $k = 3$. Then $2n=3d+6$ and $\delta = \frac{d}{2}-1$.
\item[(3)] When $k=2$. We have $n = 2(n-d-2)$, which yields $n=2d+4$, a contradiction to ${\rm Sec}(S) = \pit^{n-1}$ of Proposition \ref{p.secant}.
 \end{itemize}
 We conclude that $\delta = \frac{d}{2}-1$ is the only possibility.
\end{proof}

\section{Review of results on QEL-manifolds}
 For the proof of Theorem \ref{t.bir}, we will use the notion of QEL-manifolds introduced by Russo in \cite{R2}.
  In this section, we review some results on QEL-manifolds from \cite{IR2} and \cite{R2} that are needed  for our purpose.

\begin{definition} \label{d.qel}
Let $S \subset \pit V$ be a nondegenerate submanifold with secant defect $\delta$.
\begin{itemize}
\item[(1)] For a point $z \in {\rm Sec}(S) \setminus S$, the {\em entry locus} of $S$ with respect to $z$ is the subvariety
 $\Sigma_z(S) \subset S$ defined by $$\Sigma_z(S) = \mbox{ the closure of } \{ x \in S \ \mid \ z \in \ell_{x,y} \mbox{ for some } y \in S, y \neq x\}.$$
The cone over $\Sigma_z(S)$ with the vertex at $z$ is denoted by $C_z(S)$. If $z$ is general, then $\Sigma_z(S) = S \cap C_z(S).$
\item[(2)] $S$ is said to be a {\em quadratic entry locus manifold} (QEL-manifold in abbreviation) if for a general $z \in {\rm Sec}(S)$, the cone
$C_z(S) \subset \pit V$ is a linear subspace of dimension $\delta+1$ and the entry locus $\Sigma_z(S) \subset C_z(S)$ is a nonsingular quadratic hypersurface in the linear
subspace.
\end{itemize}

\end{definition}

This notion is relevant to us by the following.

\begin{proposition}\label{p.SQEL}
A special quadratic manifold $S \subset \BP U$ is a QEL-manifold and is linearly normal. \end{proposition}

\begin{proof}
Proposition 2.3 (b) of \cite{ES} says that a special quadratic manifold $S$ is a QEL-manifold.
Proposition 1.3 of \cite{R2} says that a QEL-manifold $S \subset \BP U$ is linearly normal if ${\rm Sec}(S) = \BP U$.
But the latter condition holds for a special quadratic manifold by Proposition \ref{p.secant}. \end{proof}

In what way is Proposition \ref{p.SQEL} useful to us? The biggest advantage of  considering QEL-manifolds is that one can use an inductive
argument via varieties of minimal rational tangents. Let us recall the definition.

\begin{definition}\label{d.vmrt}
 Let  $M \subset \pit^N$ be a prime Fano manifold. The VMRT at a point $x$ is the subvariety $\sC_x \subset \BP T_xM$ consisting of tangent directions to lines on $M$ through $x$. When $x$ is a general point of $M$, the VMRT at $x$ is nonsingular.
 \end{definition}

\begin{examples} \label{e.IHSS}
An irreducible Hermitian symmetric space of compact type is a
homogeneous space $M= G/P$ with a simple Lie group $G$ and a
maximal parabolic subgroup $P$ such that the isotropy
representation of $P$ on $T_x(M)$ at a base point $x \in M$ is
irreducible.
 The highest weight orbit of the isotropy action on $\BP T_x(M)$
is exactly the VMRT at $x$.  The following table (e.g. Section 3.1 \cite{FH}) collects basic information on these varieties.

\begin{center}
\begin{tabular}{|c| c| c| c| }
\hline Type & I.H.S.S. $M$   &  VMRT  $S$ &  $S \subset \BP
T_x(M)$
\\ \hline I    &  $ \Gr(a, a+b) $ & $\pit^{a-1} \times \pit^{b-1}$
& Segre  \\  \hline II  & $\mathbb{S}_{n}$ & $\Gr(2, n)$  & Pl\"ucker  \\
\hline III & $ \Lag(2n)$ & $\pit^{n-1}$ & Veronese \\  \hline
 IV & $\Q^n$ & $\Q^{n-2}$ & Hyperquadric \\  \hline
 V & $\mathbb{O}\pit^2$ & $\mathbb{S}_{5}$ & Spinor  \\  \hline
 VI & $E_7/(E_6 \times U(1)) $ & $\mathbb{O}\pit^2$ & Severi \\
 \hline
\end{tabular}
\end{center}

\end{examples}

\begin{examples} \label{e.SymGr}
Let $\Sigma$ be an $n$-dimensional vector space endowed with a
skew-symmetric 2-form $\omega$ of maximal rank. The symplectic Grassmannian
$M=\Gr_\omega(k, \Sigma)$ is the variety of all $k$-dimensional
isotropic subspaces of $\Sigma$, which is not homogeneous if $n$ is odd.
  Let $W$ and $Q$ be vector spaces of
dimensions $k \geq 2$ and $m$ respectively.
Let $\mathbf{t}$ be the tautological line
bundle over $\pit W$.
The VMRT $\mathcal{C}_x \subset \pit T_x(M)$ of $\Gr_\omega(k, \Sigma)$
at a general point is isomorphic to the projective
bundle $\pit((Q \otimes \mathbf{t}) \oplus \mathbf{t}^{\otimes
2})$ over $\pit W$ with the projective embedding given by
the complete linear system $$H^0(\pit W, (Q \otimes \mathbf{t}^*)
\oplus (\mathbf{t}^*)^{\otimes 2}) = (W \otimes Q)^* \oplus \SYM^2
W^*.$$
\end{examples}

An inductive argument in studying QEL-manifolds is provided by the following result from  Theorem 2.1, Theorem 2.3 and Theorem 2.8 in \cite{R2}.

\begin{proposition} [\cite{R2}]\label{p.pic}
Let $S^d \subset \BP U$ be a $d$-dimensional QEL-manifold with secant defect $\delta$.
\begin{itemize} \item[(1)] If $\delta \geq 1$ and $S$ is a prime Fano manifold,
then $d + \delta$ is even and $K_S^{-1} = \sO(\frac{d + \delta}{2})$.
\item[(2)] If $\delta \geq 3$, then   $d- \delta$ is divisible by $2^{[\frac{\delta-1}{2}]}$, the QEL-manifold $S$ is  a prime Fano manifold and the VMRT $\sC_x \subset \BP T_x(S)$ at a general point is itself a
QEL-manifold with $\dim \sC_x = \frac{d + \delta -4}{2}$ and $\delta_{\sC_x} = \delta -2$. In fact, points of $\sC_x$
corresponding to the lines through $x$ contained in an entry locus of $S$ through $x$ form an entry locus of $\sC_x$. \end{itemize}
\end{proposition}

Since in many cases the VMRT $\sC_x$ at a general $x \in S$ gives a good deal of information for $S$,
  if a QEL-manifold $S$ has $\delta \geq 3$, Proposition \ref{p.pic} enables us to reduce a problem on $S$ to a lower-dimensional QEL-manifold $\sC_x$.
  Notice that such an inductive argument is not available for special quadratic manifolds. This is why Proposition \ref{p.SQEL} is
  useful to us.
    The inductive argument for QEL-manifolds works perfectly when $\delta_S \geq \frac{1}{2} \dim S$ and we have
the following two classification results due to Russo.

\begin{proposition}[Corollary 3.1 in \cite{R2}]\label{p.3.1}
Let $S^d \subset \pit^N$ be a $d$-dimensional QEL-manifold with $\frac{d}{2}< \delta < d$. Then $S \subset \BP^N$ is isomorphic to one of the followings.
\begin{itemize}
\item[(A1)] the Segre variety $\BP^1 \times \BP^2 \subset \BP^5$;
\item[(A2)] the  Pl\"ucker embedding $\Gr(2, 5) \subset \BP^9$;
\item[(A3)] the Spinor variety $\mathbb{S}_5 \subset \BP^{15}$;
\item[(A4)] a general hyperplane section of the Pl\"ucker embedding $\Gr(2, 5) \subset \BP^9$;
\item[(A5)] a general hyperplane section of the Spinor variety $\mathbb{S}_5 \subset \BP^{15}$.
\end{itemize}
\end{proposition}

\begin{proposition}[Corollary 3.2 in \cite{R2}]\label{p.3.2}
Let $S^d \subset \pit^N$ be a $d$-dimensional QEL-manifold with $\frac{d}{2}= \delta$. Then $d=2,4,8$ or $16$ and
 $S \subset \BP^N$ is isomorphic to one of the followings.
\begin{itemize}
\item[(B1)] a general hyperplane section of the Segre variety $\BP^1 \times \BP^2 \subset \BP^5$;
\item[(B2)] the Veronese surface $v_2(\BP^2)  \subset \BP^5$;
\item[(B3)] the Segre variety $\BP^1 \times \BP^3 \subset \BP^{7}$;
\item[(B4)]  the Segre variety $\BP^2 \times \BP^2 \subset \BP^8$;
\item[(B5)] a general codimension-2 linear section of the Pl\"ucker embedding $\Gr(2, 5) \subset \BP^9$;
\item[(B6)] a general codimension-2 linear section of the Spinor variety $\mathbb{S}_5 \subset \BP^{15}$;
\item[(B7)] the Pl\"ucker embedding $\Gr(2, 6) \subset \BP^{14}$;
\item[(B8)] the $E_6$-Severi variety $\mathbb{OP}^2 \subset \BP^{26}$.
\end{itemize}
\end{proposition}

\begin{remark}  Corollary 3.2 of \cite{R2} lists one hypothetical case of $d=16$ in addition to (B8).
That case does not exist in view of Main Theorem of Section 2 in \cite{Mo}. Corollary 3.2 of \cite{R2}
 is stated for ``LQEL-manifolds", a weaker notion than QEL-manifolds,  with $\frac{d}{2} = \delta$. Because of this, the list there
  includes some biregular projections of QEL-manifolds.
Since biregular projections of QEL-manifolds can not be QEL-manifolds by Proposition 1.3 of \cite{R2}, we have the above list.
  \end{remark}

  Note that the inductive argument using Proposition \ref{p.pic} cannot be continued if $\delta \leq 2$. This is a major difficulty
  in the study of QEL-manifolds. In our case, however,  thanks to the restriction coming from Theorem \ref{t.Sato}, such low defect cases can be handled by a number of explicit
  classification results on QEL-manifolds in the following two extreme cases.

\begin{proposition}\label{p.Mukai}
Let $S^d \subset \pit^N$ be a QEL-manifold with ${\rm Sec}(S) = \pit^N$ and  $\delta \geq 2$. If $S$ is a prime Fano manifold with
$K^{-1}_S = \sO(d-2)$, then
 it is projectively equivalent to one of the following:
\begin{itemize}
\item[(M1)] the 10-dimensional Spinor variety $\mathbb{S}_5 \subset \pit^{15}$;
\item[(M2)]  a nonsingular linear section of $\mathbb{S}_5 \subset \pit^{15}$ of codimension $\leq 4$.
\end{itemize}
\end{proposition}

\begin{proof}
The assumption ${\rm Sec}(S) = \pit^N$ implies that $S \subset \pit^N$ is linearly normal by Proposition 1.3 of \cite{R2}.
By Proposition \ref{p.pic}, we have $d+\delta = 2(d-2)$, hence $\delta = d-4$ and $N = 2d +1 - \delta = d +5$.
By our assumption $\delta \geq 2$,  we obtain $d \geq 6$.

We will use Mukai's classification in \cite{M}  of linearly normal
prime Fano manifolds $S^d \subset \pit^N$ with $K^{-1}_S =
\sO(d-2)$. Mukai's classification is in terms of the genus $g:=
\frac{1}{2} {\rm deg}(S) + 1.$ If $g \leq 5$, then such $S$ is a
complete intersection by Remark 2 in \cite{M}. Since a
QEL-manifold that is a complete intersection must be a
hyperquadric by Proposition 3.4 of \cite{IR2}, we have only $g
\geq 6$. The classification for $g\geq 6$ is listed in Example
5.2.2 and Theorem 5.2.3 of \cite{IP}. In the list, the only
possibilities of $S^d \subset \BP^{d+5}$ with $d \geq 6$ occur in
Example 5.2.2 (v) and these are exactly (M1) and (M2).
\end{proof}

For the next classification result, we need the notion of rational normal scrolls.

\begin{definition}\label{d.scrolls}
For integers $a_1, \cdots, a_n \geq 1$, the {\em rational normal scroll} $S(a_1, \cdots, a_n)$ is
the projective bundle  $\pit(\0_{\pit^1}(a_1) \oplus \cdots \oplus \0_{\pit^1}(a_n))$ over $\pit^1$,
 embedded into $\pit^N$ with $N=\sum_i (a_i+1) -1$ by the
 tautological line bundle. In particular, the fibers of the
 projective bundle are mapped to linear subspaces in $\pit^N$.
\end{definition}

The following is proved in  Theorem 3 of \cite{R1}.

\begin{proposition}\label{p.OADP}
Let $S \subset \BP^{2d+1}$ be a QEL-manifold with $\delta =0$.  \begin{itemize} \item[(1)]
 When $\dim S=1$, the curve $S \subset \BP^3$ is isomorphic to the twisted cubic.
\item[(2)] When $\dim S =2$, there are three possibilities for the surface $S \subset \BP^5$:
\begin{itemize} \item[(a)] a Del Pezzo surface of degree 5 in $\pit^5$; 
\item[(b)] the rational normal scroll $S(2,2)$;
\item[(c)] the rational normal scroll $S(1,3).$
\end{itemize}  
\end{itemize}
\end{proposition}

\section{Intersection of entry loci on QEL-manifolds}\label{s.entry}
In this section, we study the intersection of entry loci on QEL-manifolds satisfying certain additional conditions. Although our main purpose  is to use this in the proof of Theorem \ref{t.bir},   our results here have  independent interest in projective geometry
of QEL-manifolds.
 The following notation will be used.

\begin{notation}
For any subset $A \subset \pit^{n-1}$, its linear span will be denoted by $\langle A \rangle \subset \pit^{n-1}$. For a point $s \in \pit^{n-1}$, we will write $\langle A, s \rangle$ in place of $\langle A \cup \{ s \}\rangle$.
For a projective variety $S \subset \pit^{n-1}$ covered by lines and a point $s \in S$, we denote by ${\rm Locus}(s) \subset S$ the locus of lines on $S$ passing through $s$. Thus for a general point $x \in S$, $$ \dim {\rm Locus}(x) = \dim \sC_x + 1.$$  If $s \in S$ is a nonsingular point, then ${\bf T}_sS \subset \pit^{n-1}$ denotes the projective tangent space of $S$ at $s$. When $S\subset \pit^{n-1}$ is a QEL-manifold with $\delta >0$,   for a  point $u \in {\rm Sec}(S)$, we recall from Definition \ref{d.qel} that $\Sigma_u(S)$ is the entry locus of $S$ and $C_u(S) = \langle \Sigma_u(S) \rangle.$
\end{notation}

\begin{condition}\label{c} We consider the following condition on $S$:
$$
 \text{For a general} \ u \in {\rm Sec}(S), \ \text{we have} \  \Sigma_t(S) =\Sigma_u(S) \ \text{for all } t \in   \langle \Sigma_u \rangle \setminus \Sigma_u.
$$
\end{condition}

Condition \ref{c} is relevant to us by the following fact from Remark 2.4 \cite{AS2}.
\begin{lemma}\label{l.AS}
A special quadratic manifold satisfies Condition \ref{c}.
\end{lemma}

\begin{proposition}\label{p.general}
Let $S \varsubsetneq \pit^{n-1}$ be a QEL-manifold with $\delta >0$ satisfying Condition \ref{c}.
Fix a general entry locus $\Sigma = \Sigma_z(S)$ by choosing a general point $z \in {\rm Sec}(S)$ and fix a general point $x \in \Sigma$. There exists a Zariski open subset $O \subset S$ with $O \cap \Sigma = \emptyset$ such that for any $s\in O$, we have the following. \begin{itemize}
\item[(1)] Either ${\rm Locus}(s) \cap \Sigma = \emptyset$ for all $s \in O$ or
 ${\rm Locus}(s) \cap \Sigma  \nsubseteq {\bf T}_x \Sigma \cap \Sigma$  for all $s \in O$.
\item[(2)] There exists  a unique entry locus  of $S$ passing through
$x$ and $ s,$ to be denoted by $\Sigma_s \in S,$ such that $\Sigma_s \subset \langle \Sigma_s \rangle$ is a smooth hyperquadric.
\item[(3)] $\langle {\bf T}_x S, s\rangle \cap S = ({\bf T}_x S \cap S) \cup \Sigma_s$.
\item[(4)] The line $\ell_{s,x}$ is not contained in $S$.
\item[(5)] For any $t\in \langle \Sigma_s \rangle \setminus \Sigma_s$, we have $C_t (S) =\langle \Sigma_s \rangle $ and $\Sigma_t(S)= \Sigma_s$.
\item[(6)] Define $P_s:=\Sigma_s \cap \Sigma$. Then $P_s$ coincides with $\langle \Sigma_s \rangle \cap \langle \Sigma \rangle$
and  is a linear subspace of $\pit^{n-1}$. In particular,  we have $P_{s}  \subset  {\bf T}_x \Sigma \cap \Sigma.$
\item[(7)] Let $p_{\Sigma}: S \setminus \Sigma \to \pit^{ n- \delta -3}$ be the projection from   $\langle \Sigma \rangle .$ Then
  $p_{\Sigma}$ is a  smooth morphism at $s$. \end{itemize}
\end{proposition}

We will use the following lemma.

\begin{lemma} [\cite{IR2}, Lemma 1.2 and Theorem 2.3] \label{l.tangential}
Let $S^d  \subset \pit^{n-1}$ be a QEL-manifold with $\delta > 0$. Then

(i) Through two general points of $S$, there passes a unique entry locus, which is a smooth hyperquadric in its linear span.

(ii) Let $x \in S$ be a general point and let
$p_x:   S \dasharrow \pit^{n-d-2}$ be the tangential projection from ${\bf T}_xS$. Then the closure of the fiber of
$p_x$ through a general point $s \in S$ is  the entry locus passing through $s$ and $x$, hence
it is a smooth hyperquadric of dimension $\delta$.
\end{lemma}

\begin{proof}[Proof of Proposition \ref{p.general}]
For each of the conditions (1)-(7), we will show that there exists a Zariski open subset of $S$ every element of which satisfies the
condition. Then $O$ can be taken as the intersection of these open subsets.

For (1), assume that ${\rm Locus}(s) \cap \Sigma$ is non-empty for a general $s \in S$.
Then $$\bigcup_{y \in \Sigma\setminus {\bf T}_x\Sigma} {\rm Locus}(y)$$ covers an open subset of $S$ and for any $s$ inside this open subset, $${\rm Locus}(s) \cap \Sigma  \nsubseteq {\bf T}_x \Sigma \cap \Sigma.$$

 (2) follows from Lemma \ref{l.tangential} (i) while (3) follows from
Lemma \ref{l.tangential} (ii).

Note that we have ${\rm Locus}(x) \neq S$ as $S \neq \pit^{n-1}$. Thus any point $s \in S \setminus {\rm Locus}(x)$ satisfies the condition (4).

(5)  is clear from Condition \ref{c}.

For (6), we take the open subset in (5).
 If  $\langle \Sigma_s \rangle \cap \langle \Sigma \rangle$
 contains a point not on $S$, say $z$, then $\langle \Sigma_s \rangle= C_z = \langle \Sigma \rangle$ by (5), which gives a contradiction.
Thus $\langle \Sigma_s \rangle \cap \langle \Sigma \rangle \subset S$, hence $\Sigma_s \cap \Sigma = \langle \Sigma_s \rangle \cap \langle \Sigma \rangle$ is a linear subspace. The inclusion $P_{s}  \subset  {\bf T}_x \Sigma \cap \Sigma$ follows from the fact that
 $P_{s} \subset \Sigma$ is linear  and $x \in P_{s}$.

(7) is obviously true for a general point $s$.
\end{proof}

\begin{proposition}\label{p.W}
In Proposition \ref{p.general}, there exists a linear subspace $W_s \subset \pit^{n-1}$ for each $s \in O$ such that
 $$s \in W_s, \ \dim W_s = \dim P_s, \ \langle P_s \cap W_s, s \rangle = W_s \mbox{ and  } \langle P_s, s \rangle \cap \Sigma_s = P_s \cup W_s.$$
 In particular, $\langle P_s, s \rangle \nsubseteq \Sigma_s$ for each $s \in O.$
  \end{proposition}

 \begin{proof} As the line $\ell_{s,x}$ is not contained in $S$ by Proposition \ref{p.general} (4),  we see that $\langle P_s, s \rangle \nsubseteq \Sigma_s$. As $\langle P_s, s\rangle$ is a linear subspace of $\langle \Sigma_s \rangle$ and $\Sigma_s$ is a quadric hypersurface in $\langle \Sigma_s \rangle$, the intersection
 $\langle P_s, s \rangle \cap \Sigma_s$ is a hypersurface of degree $\leq 2$ in  $\langle P_s, s \rangle$, which contains already the linear subspace $P_s$.
 As $s \notin P_s$,   there exists a linear subspace $W_s$ with $\dim W_s = \dim P_s$ satisfying $$ s \in W_s \mbox{ and } \langle P_s, s \rangle \cap \Sigma_s = P_s \cup W_s.$$ Moreover $P_s \cap W_s$ must be a hyperplane in $W_s$ disjoint from $s$.
 Thus $\langle P_s \cap W_s, s \rangle = W_s.$ \end{proof}

\begin{proposition} \label{p.fiber} In Proposition \ref{p.general},
assume furthermore that $S \subset \pit^{n-1}$ is defined by quadratic equations. Then
for any $s\in O \subset S$, there exists a linear subspace $F_s \subset S$ such that
\begin{equation} \label{e.fiberP}
\langle \Sigma, s \rangle  \cap S = \Sigma  \cup F_s.
\end{equation}
Moreover, we have
$F_s = \langle {\rm Locus}(s) \cap \Sigma, s \rangle$ and $\dim F_s = \dim ({\rm Locus}(s) \cap \Sigma)+1$.
In particular, $\dim F_s =0$  if and only if ${\rm Locus}(s) \cap \Sigma = \emptyset$.
\end{proposition}

\begin{proof} Let $F_s$ be the closure of the fiber of $p_{\Sigma}$ through $s$, which is nonsingular at $s$ by Proposition \ref{p.general} (7).  We claim that $F_s$ is  a cone over ${\rm Locus}(s) \cap \Sigma$ with vertex $s$.  Then $F_s$ must be linear as it is nonsingular  at $s$ and Proposition \ref{p.fiber} follows.

To prove the claim,
pick any  $s' \in (\langle  \Sigma, s \rangle \cap S) \setminus \Sigma$ with $s \neq s'$.  It suffices to show that $\ell_{s,s'} \subset S$,
 which implies $$s' \ \in \ \langle {\rm Locus}(s) \cap \Sigma, s \rangle \ \subset \
 \langle \Sigma, s \rangle \cap S, $$ proving the claim.

 As
 $\langle s, \Sigma \rangle  = \langle s', \Sigma \rangle$,  the line $\ell_{s,s'}$ intersects $\langle \Sigma \rangle$ at a single point, say $t \in \langle \Sigma \rangle$.
Suppose that $t  \notin \Sigma$.  Then $\Sigma_t(S) = \Sigma$  by Condition \ref{c}.
   As  $\ell_{s,s'}$ passes through
$t$, we have   $s, s' \in \Sigma_t(S) = \Sigma$, a contradiction.
 Thus $t \in \Sigma$, which implies that the line $\ell_{s,s'}$ has 3 intersection points with $S$. From the assumption that  $S$ is defined by quadratic equations, we have $\ell_{s,s'} \subset S$, proving the claim.
\end{proof}

\begin{proposition}\label{p.W2}
Assume that $S \subset \pit^{n-1}$ is a QEL-manifold  with
$\delta>0$ that  is defined by quadratic equations and satisfies
Condition \ref{c}. Let $O$ be as in Proposition \ref{p.general}.
Then for any $s\in O$, we have
\begin{equation} \label{e.fiberT}
\langle {\bf T}_x \Sigma, s \rangle \cap S = ({\bf T}_x\Sigma \cap S) \cup W_s
\end{equation} where $W_s$ is as in Proposition \ref{p.W}.
\end{proposition}

It is convenient to recall the following straightforward lemma.

\begin{lemma}\label{l.linear}
Let $L_1,  L_2 \subset \pit^{n-1}$ be two linear subspaces. Then \begin{itemize}
\item[(i)] $\langle L_1 \cap L_2, s \rangle = \langle L_1, s \rangle \cap L_2$ if $s \in L_2$; and \item[(ii)]
$\langle L_1, s \rangle \cap L_2 = L_1 $ if $L_1 \subset L_2$ and $s \not\in L_2$. \end{itemize} \end{lemma}

\begin{proof}[Proof of Proposition \ref{p.W2}]
To start with, we claim the following relation
\begin{equation}\label{e.w1} \langle P_{s}, s \rangle \cap \Sigma_{s}   \subset  \langle {\bf T}_x \Sigma, s \rangle \cap \Sigma_{s}   \subsetneq \langle \Sigma, s \rangle  \cap \langle \Sigma_s \rangle  = \langle P_{s}, s\rangle. \end{equation}
The first inclusion is from  $P_{s}  \subset  {\bf T}_x \Sigma$ in Proposition \ref{p.general} (6).  The second inclusion is from
 $ {\bf T}_x \Sigma \subset \langle \Sigma,s \rangle$. The last equality follows from  $$\langle P_{s}, s\rangle = \langle \langle \Sigma \rangle \cap \langle \Sigma_s \rangle, s \rangle = \langle \Sigma, s \rangle \cap \langle \Sigma_s \rangle$$ which is a consequence of Proposition \ref{p.general} (6) and Lemma \ref{l.linear} (i). Finally, the inequality in the middle is by
$\langle P_s, s \rangle \not\subset \Sigma_s$ from Proposition \ref{p.W}.

 In \eqref{e.w1},  the proper subvariety  $\langle {\bf T}_x \Sigma, s \rangle \cap \Sigma_s$ of $\langle P_s, s \rangle$
 is a subvariety of degree $\leq 2.$ Since $\langle P_s, s \rangle$ is a linear subspace of $\langle \Sigma_s \rangle$ from Proposition \ref{p.general} (6)
 and $\Sigma_s \subset \langle  \Sigma_s \rangle$ is a quadric hypersurface from Proposition \ref{p.general} (2), the intersection
 $\langle P_s, s \rangle \cap \Sigma_s$ must be a hypersurface of degree 2 in $\langle P_s, s \rangle$.
It follows that  \begin{equation}\label{e.w2} \langle P_{s}, s \rangle \cap \Sigma_{s}  = \langle {\bf T}_x \Sigma, s \rangle \cap \Sigma_{s}. \end{equation}
From ${\bf T}_x \Sigma \subset {\bf T}_x S$, we have the tautological relation
 \begin{equation}\label{e.w3} \langle {\bf T}_x \Sigma, s\rangle  \cap S = \langle {\bf T}_x \Sigma, s \rangle
  \cap \left( \langle {\bf T}_x S, s \rangle \cap S \right) \end{equation}
  Combining \eqref{e.w3} with
$\langle {\bf T}_x S, s \rangle \cap S =  \Sigma_{s} \cup ({\bf T}_xS \cap S)  $ from Proposition \ref{p.general} (3), we obtain \begin{equation}\label{e.w4}
   \langle {\bf T}_x \Sigma, s\rangle  \cap S
   =  (\langle {\bf T}_x \Sigma, s\rangle \cap \Sigma_{s}) \cup (\langle {\bf T}_x \Sigma, s \rangle \cap {\bf T}_xS \cap S). \end{equation}
For the first term on the righthand side of \eqref{e.w4}, we have
\begin{equation}\label{e.w5}
\langle {\bf T}_x \Sigma, s \rangle \cap \Sigma_{s} = \langle P_{s}, s\rangle \cap \Sigma_{s} = P_{s} \cup W_{s}
\end{equation}
where the first equality is from \eqref{e.w2} and the second equality is from Proposition \ref{p.W}.
For the second term on the right hand side of \eqref{e.w4}, we have
 \begin{equation}\label{e.w6} \langle {\bf T}_x \Sigma, s\rangle  \cap {\bf T}_xS \cap S = {\bf T}_x \Sigma \cap S \end{equation}
because $\langle {\bf T}_x \Sigma, s \rangle \cap {\bf T}_x S =
{\bf T}_x \Sigma$ by Lemma \ref{l.linear} (ii). Putting
\eqref{e.w5} and \eqref{e.w6} into \eqref{e.w4},  we  obtain
$$\langle {\bf T}_x \Sigma, s\rangle \cap S =  P_s \cup W_{s} \cup  ({\bf T}_x \Sigma \cap S). $$ Recalling $P_s \subset {\bf T}_x \Sigma \cap \Sigma$ from Proposition \ref{p.general} (6), we have \eqref{e.fiberT}.
\end{proof}

\begin{theorem} \label{t.fiberdim}
In Proposition \ref{p.W2}, assume that $S$ is prime Fano and that  ${\rm Locus}(s) \cap \Sigma \neq \emptyset$ for $s \in O$.  Then \begin{itemize}
\item[(i)] $\dim F_s = \frac{3 \delta - d}{2}$; \item[(ii)] $W_s$ is a codimension-1 linear subspace of $F_s$; and \item[(iii)]
  $\Sigma \cap \Sigma_{s} \simeq \BP^{\frac{3\delta-d-2}{2}}$. \end{itemize}
\end{theorem}
\begin{proof}
Consider the incidence variety
$$
I = \mbox{ the closure of } \{(t, y) \in \Sigma \times O |
\ell_{ty} \subset S\}
$$
and the two projections $p_1: I \to \Sigma$ and $p_2: I \to S$.
As $S$ is a prime Fano QEL-manifold,  Proposition \ref{p.pic} gives $\dim \sC_t = \frac{d + \delta}{2}-2$ for a general $t \in  S$, which implies
$$\dim {\rm Locus}(t)= \dim \sC_t +1 = \frac{d+\delta}{2}-1.$$ Since we have chosen our $\Sigma$ generally,
    the  dimension of a general fiber of $p_1$ is equal to $\dim {\rm Locus}(t)$ for general $t \in S$.  Thus  $$\dim I - \dim \Sigma = \frac{d+\delta}{2}-1 \mbox{ and } \dim I = \delta +  \frac{d+\delta}{2}-1.$$
     The assumption ${\rm Locus}(s) \cap \Sigma \neq \emptyset$ for $s \in O$ implies that $p_2$ is surjective by Proposition \ref{p.general} (1). Thus the dimension of a general fiber of $p_2$ is $$\dim I - \dim S = \delta +  \frac{d+\delta}{2} -1 -d =   \frac{3 \delta-d}{2}-1.$$
     Since the fiber of $p_2$ over a general $s \in O$ is ${\rm Locus}(s) \cap \Sigma$  and $F_s = \langle {\rm Locus}(s) \cap \Sigma, s \rangle$ from Proposition \ref{p.fiber}, we have $\dim F_s = \frac{3 \delta - d}{2}$, proving (i).

    Comparing  equations \eqref{e.fiberP} and \eqref{e.fiberT}, we see that $W_s \subset F_s$ is a linear subspace  of codimension $\leq 1$.
Suppose that $W_s = F_s$.
Note that  $ P_{s} \cap W_{s} \subset {\rm Locus}(s) \cap \Sigma$ from $P_s \subset \Sigma$ and $W_s \subset {\rm Locus}(s)$.
As $F_s = \langle s,{\rm Locus}(s) \cap \Sigma\rangle$ from Proposition \ref{p.fiber}  and  $ W_{s} =  \langle s, P_{s} \cap W_{s} \rangle$ from Proposition \ref{p.W}, the hypothesis $F_s = W_s$ implies ${\rm Locus}(s) \cap \Sigma = P_{s} \cap W_{s}.$
Combining this with Proposition \ref{p.general} (6), we have $$ {\rm Locus}(s) \cap \Sigma \subset P_{s} \subset \Sigma \cap {\bf T}_x \Sigma$$which contradicts  Proposition \ref{p.general} (1). Thus $W_{s}$ is of codimension one in $F_s$, proving (ii).

 Finally, the equality $\dim P_s = \dim W_s$ in Proposition \ref{p.W} implies  $$\dim (\Sigma \cap \Sigma_s) = \dim W_s = \dim F_s -1 = \frac{3 \delta -d -2}{2},$$ proving (iii).
\end{proof}

We close this section with one explicit example of the intersection of entry loci.
It is a simple consequence of the following elementary lemma, whose proof will be skipped.

\begin{lemma}\label{l.S}
Let $V$ be a vector space of dimension $2m \geq 6,$ equipped with a non-degenerate quadratic form $Q: V \times V \to \C$. \begin{itemize} \item[(1)]
The two components $\mathbb{S}^+(V;Q)$ and $\mathbb{S}^-(V; Q)$ of the space of $m$-dimensional $Q$-isotropic subspaces of $V$ admit  projective embeddings, realizing them as prime Fano manifolds biregular to each other.  \item[(2)] Let $W \subset V$ be a $Q$-isotropic subspace of dimension $\leq m-3$ and let
$$W^{\perp} := \{ v \in V, \ Q(v,w) = 0 \mbox{ for all } w \in W.\}.$$  The quotient space $W^{\perp}/W$ with $\dim W^{\perp}/W \geq 6$ is equipped with a non-degenerate quadratic from $Q^W$ induced by $Q$  and  we have  natural inclusions $$\mathbb{S}^+(W^{\perp}/W; Q^W) \subset \mathbb{S}^+(V; Q) \mbox{ and } \mathbb{S}^-(W^{\perp}/W; Q^W) \subset \mathbb{S}^-(V; Q)$$
which induce  isomorphisms between their Picard groups. \item[(3)] When $W = \C v$ for some $0 \neq v \in V$, let us write $\mathbb{S}^+(W^{\perp}/W; Q^W)$ as  $\mathbb{S}^v.$ For two independent vectors $v, v' \in V$, if $Q(v, v')\neq 0,$ then  $\mathbb{S}^v \cap \mathbb{S}^{v'} = \emptyset$, while if $Q(v,v') =0,$ then  $$\mathbb{S}^v \cap \mathbb{S}^{v'} = \mathbb{S}^+(W^{\perp}/W; Q^W) \mbox{ for } W = \C v + \C v'.$$ \end{itemize} \end{lemma}

We denote by $\mathbb{S}_m$ the Spinor variety $\mathbb{S}^+(V; Q)$ of Lemma \ref{l.S}.
 We have the following consequence.

\begin{proposition}\label{p.intS}
The intersection of any two distinct entry loci of $\mathbb{S}_5 \subset \pit^{15}$ is either empty or  $\pit^3$.
\end{proposition}

\begin{proof}
It is well-known that the embedding of $\mathbb{S}_m$ as a prime Fano manifold in Lemma \ref{l.S} becomes an isomorphism $\mathbb{S}_3 \cong \pit^3$ when $m=3$ and an isomorphism $\mathbb{S}_4 \cong \mathbb{Q}^6 \subset \pit^7$ when $m=4$.
Applying Lemma \ref{l.S} with $m=5$, we have a natural inclusion $\mathbb{S}_4 \cong \mathbb{S}^v \subset \mathbb{S}_5$ for any isotropic vector
 $v \in V$. As ${\rm Aut}(\mathbb{S}_5)$ acts transitively on
 $\BP^{15} \setminus \mathbb{S}_5$, every entry locus of
 $\mathbb{S}_5$ is isomorphic to $\qit^6$. Conversely, any
 $\qit^6$ on $\mathbb{S}_5$ is contained in the entry locus of any
 point in $\langle \qit^6 \rangle$, hence it is an entry locus.
 This implies that for any isotropic $v \in V$, $\mathbb{S}^v$ is
 an entry locus of $\mathbb{S}_5$.

By Example \ref{e.table} and Proposition \ref{p.blow}, we have a
$\BP^7$-bundle
$$\psi: {\rm Bl}_{\mathbb{S}_5}(\BP^{15}) \to
\qit^8.$$
 Any fiber of $\psi$ is mapped to the linear span of an
entry locus of $\mathbb{S}_5$, hence the set of all entry loci of
$\mathbb{S}_5$ is parametrised by $\qit^8$, hence
 every entry locus of $\mathbb{S}_5$ is of the form $\mathbb{S}^v$ for some isotropic vector $v$.
 Take any two distinct entry loci $\mathbb{S}^v$ and $\mathbb{S}^{v'}$. If their intersection is non-empty, then $Q(v, v')=0$ and  $\mathbb{S}^v \cap \mathbb{S}^{v'} \simeq \mathbb{S}_3 \simeq \pit^3$.
\end{proof}

Recall  (e.g. Proposition 1.3 (ii) in \cite{IR})  that the entry loci of a linear section with codimension $\leq \delta$ of a QEL-manifold $S$ with $\delta \geq 1$   come from the corresponding linear section of entry loci of $S$. Thus Proposition \ref{p.intS} implies
the following.

\begin{corollary}\label{c.intersect2}
The intersection of any two distinct entry loci of a nonsingular
linear section of $\mathbb{S}_5$ with codimension $t \leq 2$ is
either empty or isomorphic to $\BP^{k}$ with $k\geq 3-t$.
\end{corollary}

\section{Proof of Theorem \ref{t.bir}}\label{s.proof}

In this section, we will prove Theorem \ref{t.bir} by showing that a pair $(S \subset \pit U, Y \subset \pit W)$ determined by a special system of quadrics must be one of the examples in Example \ref{e.table},
Example \ref{e.projection},  Proposition \ref{p.ex1} or Proposition \ref{p.ex2}.
We will divide our argument into a number of subsections. In the first two subsections, we will classify a certain type of special quadratic manifolds. In the next three subsections, we carry out the classification of the pair $(S, Y)$.

\subsection{Classification when $\delta \geq \frac{1}{2} \dim S$}

As an easy application of Proposition \ref{p.3.1} and Proposition \ref{p.3.2}, we have the following.

\begin{proposition} \label{p.d/2}
A $d$-dimensional special quadratic manifold $S^d \subset \pit U$ with secant defect $\delta \geq \frac{d}{2}$ is projectively equivalent to one of the following.  For our later use, we specify one possibility of $Y= Y(\sigma)$ in each case in the list.  \begin{itemize}
\item[(i)] the smooth hyperquadric $\mathbb{Q}^d \subset \pit^{d+1}$ with $Y$ a point;
\item[(ii)] the Segre threefold $\pit^1 \times \pit^2 \subset \pit^5$ with $Y=\pit^2$;
\item[(iii)] the Pl\"ucker embedding of ${\rm Gr}(2, 5) \subset \pit^9$ with $Y=\pit^4$;
\item[(iv)] the 10-dimensional Spinor variety $\mathbb{S}_5 \subset \pit^{15}$ with $Y =\qit^8$;
\item[(v)] a general hyperplane section of (iii) with $Y=\pit^4$;
\item[(vi)] a general hyperplane section of (ii) with $Y=\pit^2$;
\item[(vii)] a general codimension 2 linear section of (iii) with $Y=\pit^4$;
\item[(viii)] the Segre fourfold $\pit^1 \times \pit^3 \subset \pit^7$ with $Y={\rm Gr}(2,4) \simeq \qit^4$.
\end{itemize}
\end{proposition}

\begin{proof}
If $\delta=d$,  the fact ${\rm Sec}(S) = \BP U$ and
 $\delta = 2d + 1 - \dim {\rm Sec}(S)$ implies that $S$ is a hyperquadric. This gives case (i).

As $S$ is a QEL-manifold by Proposition \ref{p.SQEL}, if $d> \delta \geq d/2$, we can apply Proposition \ref{p.3.1} and Proposition \ref{p.3.2}.
In Proposition \ref{p.3.1},  (A5) is excluded by Proposition \ref{p.AS2}.
Thus only (A1)-(A4) are possible, which gives (ii)-(v).

In Proposition \ref{p.3.2}, the four Severi varieties (B2), (B4), (B7) and (B8) do not satisfy ${\rm Sec}(S) = \BP U$ of
Proposition \ref{p.secant}. Also
(B6) is excluded by Proposition \ref{p.AS2}.
Thus only (B1), (B5) and (B3) are possible, which gives (vi)-(viii).

The example of $Y(\sigma)$ in each case follows from the table in Example \ref{e.table}, Proposition \ref{p.ex1} and Proposition \ref{p.ex2}. \end{proof}

\subsection{Classification when $\delta \geq 1$ and $S$ is not a prime Fano manifold}
Special quadratic manifolds with $\delta \geq 1$ that are not prime Fano manifolds can be classified as follows.

\begin{proposition} \label{p.CC}
Let $S \subset \BP U$ be a special quadratic manifold with $\delta \geq 1$. If $S$ is not a prime Fano manifold, then
 $S \subset \pit U$
is projectively equivalent to one of the following:
\begin{itemize}
\item[(c1)] a smooth conic in $\pit^2$;
\item[(c2)] a general hyperplane section of the Segre 3-fold $\pit^1 \times \pit^2 \to \pit^5$;
\item[(c3)] the Segre variety $\pit^1 \times \pit^{d-1} \subset \pit^{2d-1}$.
\end{itemize}
\end{proposition}

The proof is a direct consequence of the following result of Ionescu and Russo.

\begin{theorem}[Theorem 2.2 in \cite{IR}]\label{t.IR}
Let $S \subset \BP U$ be a nondegenerate, linearly normal, conic-connected manifold of dimension
$d$. Then either $S$ is a prime Fano manifold or it is projectively equivalent
to one of the following:
\begin{itemize} \item[(C1)] the second Veronese embedding of $\BP^d$.
\item[(C2)] the Segre embedding of $\BP^a \times \BP^{d-a}$ for $1
\leq a \leq d-1$. \item[(C3)] ${\rm Bl}_{\BP^k}(\BP^d)$ embedded
in $\BP^N$ by the linear system of quadric hypersurfaces of
$\BP^d$ passing through $\BP^k, 0 \leq k \leq d-2$, where $N =
d(d+3)/2-\binom{k+2}{2}$.

\item[(C4)] a hyperplane section of the Segre
embedding $\BP^a \times \BP^{d+1-a}$ with $2 \leq a, d+1-a$.
\end{itemize}
\end{theorem}

\begin{proof}[Proof of Proposition \ref{p.CC}]
By Proposition \ref{p.SQEL}, we know that $S \subset \pit U$ is linearly normal.
By Theorem 2.1 of \cite{R2}, the condition $\delta \geq 1$ implies that $S$ is conic-connected.
Thus if $S$ is not a prime Fano manifold, it must be one of (C1)-(C4) in Theorem \ref{t.IR}.

Recall that we have ${\rm Sec}(S) = \pit U$ by Proposition \ref{p.secant}.
It is well known (e.g. the table in p.466 of \cite{FH}) that   only (c1) in (C1)
and (c2) in (C2) satisfy ${\rm Sec}(S) = \pit U$ and no cases in (C4)
satisfy ${\rm Sec}(S) = \pit U.$
In (C3), only a general hyperplane section of $\pit^1 \times \pit^{d} \subset \pit^{2d+1}$
satisfies ${\rm Sec}(S) = \pit U$ by Lemma 4.19 \cite{FH}.   By Proposition \ref{p.AS2},
only (c3) can occur.
\end{proof}

\subsection{Classification when $Y = \pit^{2(c-1)}$}

We can classify special quadratic manifolds in the case of (Y1)  of Theorem \ref{t.Sato}:

\begin{proposition} \label{p.Pn}
Assume that $Y = \pit^{2(c-1)}$, then $S \subset \BP^{n-1}$ is projectively equivalent to one of the following:
\begin{itemize}
\item[(i)] $\Q^{d} \subset \BP^{d+1}$;
\item[(ii)] $\BP^1 \times \BP^2 \subset \BP^5$;
\item[(iii)] ${\rm Gr}(2, 5) \subset \BP^9$;
\item[(iv)] a general hyperplane  section of (ii);
\item[(v)] a general  codimension $\leq 2$ linear section of (iii);
\item[(vi)] a general codimension-3 linear section of (iii);
\item[(vii)] a general codimension-2 linear section of (ii).
\end{itemize} All of the above cases do occur with $Y = \pit^{2(c-1)}$ from
Example \ref{e.table}, Proposition \ref{p.ex1} and Proposition \ref{p.ex2}.
\end{proposition}
\begin{proof}
As  $Y= \pit^{2(c-1)}$, we have $\chi(Y) = 2c-1$. Proposition \ref{p.Euler}  yields $\chi(S) = 2 \delta+2$.
By  Corollary \ref{c.Betti}, we obtain $2 \delta+2 \geq d+1$, i.e., $\delta \geq (d-1)/2$.
If $\delta \geq d/2$, then  Proposition \ref{p.d/2} gives (i)-(v).
It remains to handle the case $\delta=(d-1)/2$ and $\chi(S) = d+1$.
The latter condition combined with Corollary \ref{c.Betti} implies that all even Betti numbers of $S$  must be 1.
We will argue case-by-case, depending on the values of $\delta$.

\begin{itemize}\item[(1)]
If  $\delta \geq 3$, then $d-\delta = \delta+1$ is divisible by $2^{[\frac{\delta-1}{2}]}$ by Proposition  \ref{p.pic}. This implies that either  $(\delta, d, n) = (3, 7, 13)$ or   $(\delta, d, n) = (7, 15, 25)$. \begin{itemize}
\item[(1a)] When $(\delta, d, n) = (3, 7, 13)$,  Proposition \ref{p.pic} shows that $S$ is a prime Fano manifold with
$K^{-1}_S = \sO(\frac{d+\delta}{2}) = \sO(5)$. By Proposition \ref{p.Mukai}, the only possibility is (M2), a linear section of
$\mathbb{S}_5$. But such a case can not occur by Proposition \ref{p.AS2}.
\item[(1b)]  When $(\delta, d, n)=(7,15,25)$, we see that  $S$ is a prime Fano manifold
and  its VMRT $\mathcal{C}_x \subset \pit^{14}$  at a general point $x \in S$ is a QEL-manifold
of dimension $ (d+\delta)/2-2 = 9$ and secant defect $\delta -2 =5$  from Proposition \ref{p.pic}.  It follows that the subvariety ${\rm Locus}(x) \subset S$
   has dimension 10 and has nonempty intersection with a general entry locus $\Sigma \subset S$ of dimension 7, because all
   even Betti numbers of $S$ are 1. Since $S$ is defined by quadrics and satisfies Condition \ref{c} in Section \ref{s.entry} by Lemma
   \ref{l.AS}, we can apply Theorem \ref{t.fiberdim} to conclude that two general entry loci $\Sigma_1$ and $\Sigma_2$ of $S$ through a general point $x \in S$ intersect along $\pit^2$.
   By Proposition \ref{p.pic}, the lines on $\Sigma_1$ and $\Sigma_2$ through $x$ give two entry loci on $\mathcal{C}_x \subset \pit^{14}.$ The intersection of these entry loci of $\mathcal{C}_x$ must be $\pit^1$ corresponding to
   the lines through $x$ in $\Sigma_1 \cap \Sigma_2 \cong \pit^2$.    On the other hand, applying Proposition \ref{p.pic}  to $\sC_x$, we see that
$\sC_x$ is a prime Fano manifold with $K^{-1}_{\sC_x} = \sO(7)$.
  Proposition \ref{p.Mukai} shows that
 $\mathcal{C}_x \subset \pit^{14}$ is projectively equivalent to a codimension 1 linear section of $\mathbb{S}_{5} \subset \pit^{15}$.   Thus the intersection of any two distinct entry loci of $\mathcal{C}_x$ must be either empty or $\pit^k$ with $k\geq 2$, from Corollary \ref{c.intersect2}, a contradiction.
\end{itemize}

\item[(2)] If $\delta=2$, then $d=5$ and $n=10$.  As $d+\delta$ is odd,
$S$ is not a prime Fano manifold by Proposition \ref{p.pic}. Applying Proposition \ref{p.CC}, we see that $S \subset \pit^9$ is the Segre 5-fold
$\pit^1 \times \pit^4 \subset \pit^9$, a contradiction to $\chi(S)=d+1$.
\item[(3)] If $\delta=1$, then $d=3$ and $n=7$.  Since none of the varieties in the list in Proposition \ref{p.CC} have $(d,n)=(3,7)$, we see that
 $S$ is a prime Fano manifold.   Thus  $K^{-1}_S = \sO(2)$ by Proposition \ref{p.pic} (1).
From the classification of prime Fano threefolds with $K^{-1}_S = \sO(2)$ (e.g. 12.1 in \cite{IP}),  $S$ is a general codimension-3 linear section of ${\rm Gr}(2, 5) \subset \pit^9$, which gives case (vi).
\item[(4)] If $\delta=0$, then $d=1$ and $n=4$. Thus $S$ is a twisted cubic by Proposition \ref{p.OADP}, which is (vii).
\end{itemize}
\end{proof}

\subsection{Classification when $Y = \mathbb{Q}^{2(c-1)}$}

We can classify special quadratic manifolds in the case of (Y2)  of Theorem \ref{t.Sato}:

\begin{proposition} \label{p.Y=Q}
For a special quadratic manifold $S \subset \pit^{n-1}$ of codimension $c$,  assume that $Y = \qit^{2(c-1)}$. Then $S \subset \pit^{n-1}$
is projectively equivalent  to one of the following
\begin{itemize} \item[(i)] the Segre 4-fold $\pit^1 \times \pit^3 \subset \pit^7$;
\item[(ii)]  the 10-dimensional Spinor variety $\mathbb{S}_5 \subset \pit^{15}$. \end{itemize}
Both cases do occur with $Y = \qit^{2(c-1)}$ from
Example \ref{e.table}.
\end{proposition}
\begin{proof}
If $\delta \geq \frac{d}{2}$, then Proposition \ref{p.d/2} gives (i) and (ii).

 To show that $\delta < \frac{d}{2}$ cannot occur,   we may assume $d= 2 \delta + 2$ and $n= 2d+2- \delta= 3\delta + 6$ by Proposition \ref{p.Q_delta} and derive a contradiction.
Under this assumption, putting $\chi(\qit^{2(c-1)}) = 2c$ in Proposition \ref{p.Euler}, we have $\chi(S) = d+1$, which combined with Corollary \ref{c.Betti} implies that all even Betti numbers of $S$  must be 1.  We will argue case-by-case depending on the value of $\delta$.

\begin{itemize} \item[(1)]
Assume $\delta \geq 3.$
 Proposition \ref{p.pic} shows that the VMRT $\mathcal{C}_x \subset \pit^{d-1}$  at a general point $x \in S$  is  a QEL-manifold  of dimension $\frac{1}{2}(d+ \delta -4) = \frac{1}{2}(3 \delta -2)$ and secant defect $\delta -2$.
Since $d- \delta = \delta+2$ is divisible by  $2^{[\frac{\delta-1}{2}]}$ from
 Proposition \ref{p.pic}, we have two possibilities: $(\delta, d, n)  = (4, 10, 18)$ or  $(\delta, d, n) = (6, 14, 24)$.
\begin{itemize} \item[(1a)]
 When $(\delta, d, n) = (4, 10, 18)$, the VMRT $\mathcal{C}_x \subset \pit^9$ is a QEL manifold  of dimension 5 and secant defect 2.  By Proposition  \ref{p.pic}, this implies that  $\mathcal{C}_x \subset \pit^9$
is a not a prime Fano manifold. Thus Proposition \ref{p.CC} shows  that  $\mathcal{C}_x \subset \pit^9$ is projectively equivalent to the Segre 5-fold $\pit^1 \times \pit^4 \subset\pit^9$.   By Main Theorem of Section 2 in \cite{Mo}, $S$ is isomorphic to ${\rm Gr}(2, 7)$.  But $S  \subset \pit^{17}$ is linearly normal by Proposition \ref{p.SQEL}, which is not possible for $S = {\rm Gr}(2, 7)$.

\item[(1b)] When $(\delta, d, n) = (6, 14, 24)$,
we see that  $S$ is a prime Fano manifold
and  its VMRT $\mathcal{C}_x \subset \pit^{13}$  at a general point $x \in S$ is a QEL-manifold
of dimension 8 and secant defect 4  from Proposition \ref{p.pic}.  It follows that the subvariety ${\rm Locus}(x) \subset S$
   has dimension 9 and has nonempty intersection with a general entry locus $\Sigma \subset S$ of dimension 6, because all
   even Betti numbers of $S$ are 1. Since $S$ is defined by quadrics and satisfies Condition \ref{c} in Section \ref{s.entry} by Lemma
   \ref{l.AS}, we can apply Theorem \ref{t.fiberdim} to conclude that two general entry loci $\Sigma_1$ and $\Sigma_2$ of $S$ through a general point $x \in S$ intersect along $\pit^1$.
   By Proposition \ref{p.pic}, the lines on $\Sigma_1$ and $\Sigma_2$ through $s$ give two entry loci on $\mathcal{C}_x \subset \pit^{13}.$ The intersection of these entry loci of $\mathcal{C}_x$ must be one point corresponding to
   the unique line through $x$ in $\Sigma_1 \cap \Sigma_2$.    On the other hand, applying Proposition \ref{p.pic}  to $\sC_x$, we see that
$\sC_x$ is a prime Fano manifold with $K^{-1}_{\sC_x} = \sO(6)$.
  Proposition \ref{p.Mukai} shows that
 $\mathcal{C}_x \subset \pit^{13}$ is projectively equivalent to a codimension 2 linear section of $\mathbb{S}_{10} \subset \pit^{15}$.   Thus the intersection of any two distinct entry loci of $\mathcal{C}_x$ must be either empty or $\pit^k$ with $k\geq 1$ from Corollary \ref{c.intersect2}, a contradiction.
\end{itemize}

\item[(2)] If $\delta =2$, then  $d=6$ and $n=12$. \begin{itemize} \item[(2a)] If $S$ is not a prime Fano manifold, then
Proposition \ref{p.CC} shows that  $S$ is the Segre 6-fold $\pit^1 \times \pit^5 \subset \pit^{11}$.
This  contradicts $\chi(S) = d+1$.
\item[(2b)] If $S$ is a prime Fano manifold, then Proposition \ref{p.Mukai} is applicable. The only possibility is $s=4$ of (M2), i.e.,
 a codimension-4 linear section of the 10-dimensional
Spinor variety $\mathbb{S}_5 \subset \pit^{15}$. But this can not be a special quadratic manifold by Proposition \ref{p.AS2}.
\end{itemize}
\item[(3)] If  $\delta=1$, then $d= 4$ and $n=9$.  As $d+ \delta$ is odd, Proposition \ref{p.pic} shows that $S$ is not a prime Fano manifold.
This is not possible by Proposition \ref{p.CC}.
\item[(4)]
  If $\delta =0$, then  $d=2$ and  Proposition \ref{p.OADP} is applicable. But the surfaces in Proposition \ref{p.OADP} do not have
  Picard number 1, a contradiction.
\end{itemize} \end{proof}

\subsection{Classification when $Y = \Gr(2,c+1)$}

To handle the case  (Y3) of Theorem \ref{t.Sato}, we use the following result.

\begin{proposition}[Theorem 2.17 ($b=2$) \cite{AS2}] \label{p.Bir2Gr}
Let $V$ be a vector space with $\dim V = 2k+1 \geq 7$. Let $\mu: \SYM^2 V \to W$ be a system of quadrics such that the base locus subscheme $B(\mu) \subset \pit V$ is irreducible, nonsingular and nondegenerate. Let $\nu: \pit V \dasharrow \pit W$ be the associated rational map. Suppose that $\nu$ is generically injective and its proper image  ${\rm Im}(\nu)$ is biregular to   $\Gr(2, k+2)$.
Then $B(\mu) \subset \pit V$ is one of the rational normal scrolls. In particular, $B(\mu)$ has Picard number 2 and is covered by lines.
\end{proposition}

We can classify special quadratic manifolds in the case of (Y3)  of Theorem \ref{t.Sato}:

\begin{proposition} \label{p.Gr}
  Let $S \subset \pit^{n-1}$ be a special quadratic manifold of codimension $c \geq 4$ with   $ Y \subset \pit W$ being an isomorphic projection of the Pl\"ucker variety ${\rm Gr}(2, c+1) \subset \pit (\wedge^2 \cit^{c+1})$.
   Then $S \subset \pit^{n-1}$ is projectively equivalent to
$\BP^1 \times \BP^{c} \subset \BP^{2c+1}$.
\end{proposition}

\begin{proof}
Let $\sigma: \SYM^2 U \to W, \dim U =n,$ be the special system of quadrics such that our $S$   is the base locus  $B(\sigma) \subset \pit U$. Recall that a general fiber of $\psi: \Bl_S(\BP U)  \to Y(\sigma)\cong \Gr(2, c+1)$ is sent to a linear space of dimension $\delta+1$ in $\BP U$ by Proposition \ref{p.linear}.
Take a general subspace $V \subset U$ with $$\dim U - \dim V = \delta+1, \  \dim V = 2c-1 \geq 7$$ and denote by  $S' \subset \pit V$ the linear section  $S\cap \pit V$. Then the restriction $$\psi|_{\Bl_{S'}(\BP V)} : \Bl_{S'}(\BP V) \to Y(\sigma)$$
is birational. The  birational map $\BP V \dasharrow Y(\sigma)$ is given by a system of quadrics $\mu: \SYM^2 V \to W$
induced by the restriction of  $\sigma$.
 Since the base locus $B(\mu)$ is $S'$, which is irreducible, nonsingular and nondegenerate in $\pit V$, we can apply
 Proposition \ref{p.Bir2Gr} to see that $S'$ is a rational normal scroll.
We will argue case-by-case as follows.
\begin{itemize} \item[(1)] When $\delta =0$.    Putting $n= 2d +2$ and  $\chi(\Gr(2, c+1)) = \frac{c(c+1)}{2}$ in Proposition \ref{p.Euler}, we have
$\chi(S) = d+1$. This implies that $S$ has Picard number 1.
\begin{itemize} \item[(1a)] If $c \geq 5$, then $ \dim S' = c-2   \geq 3$. Thus $S'$ and $S$ have the same Picard number by Lefschetz. Since  $S'$ has Picard number 2, this gives a contradiction. \item[(1b)] If $c=4$, then $d=3$. We conclude that $S \subset \pit^7 $ is a Fano threefold of Picard number 1 with $\chi(S) =4$. Furthermore, a general hyperplane section $S'$ of $S$ is a rational normal scroll. In particular, $S$ is covered by lines, which implies that $K^{-1}_S = \sO(k), k \geq 2$. From the classification of Fano threefolds (e.g. 12.1 of  \cite{IP}), there are three possibilities for such
 Fano threefolds: $\pit^3, \qit^3$ and a codimension-3 linear section of $\Gr(2,5) \subset \pit^9$. But none of these can have
 $\delta =0.$ \end{itemize}
\item[(2)] When  $\delta \geq 1.$
\begin{itemize} \item[(2a)] Assume that  $S$ is a prime Fano manifold and $ c \geq 5$. Then $\dim S'=c-2 \geq 3$, which implies that $S'$ and $S$ have the same Picard number  by Lefschetz. Since $S'$ has Picard number 2, this gives a contradiction. \item[(2b)] Assume that $S$ is a prime Fano manifold and $c=4.$ Then $n = d+5$
and $d - \delta =3$ is odd. This is a contradiction to Proposition \ref{p.pic}.
\item[(2c)] Assume that $S$ is not a prime Fano manifold.
Proposition \ref{p.CC}  shows that $S \subset \pit^{n-1}$ is projectively equivalent to the Segre embedding $\pit^1 \times \pit^c \subset \pit^{2c+1}$.
\end{itemize} \end{itemize} \end{proof}

\section{Quadratically symmetric varieties and prolongations}\label{s.prolongation}

In this section, we give an intrinsic characterization of $Z(\sigma)$ and relate it to the problem studied in \cite{FH}.
Throughout this section,  $X \subset \BP V$ denotes an
$n$-dimensional irreducible nondegenerate projective subvariety.

\begin{definition}\label{d.QSV}
Let $X \subset \BP V$ be an
$n$-dimensional irreducible nondegenerate projective subvariety.
We say that $X$ is a {\em quadratically
symmetric variety} if there exists a Zariski open subset $X^o \subset {\rm Sm}(X)$
in the nonsingular locus of $X$ such that for each $P \in X^o$, there exits a linear representation of the $\cit^*$-group denoted by $E_P: \cit^* \times V \to V$ on $V$ satisfying the following conditions.
 \begin{itemize} \item[(i)] Denoting by $\widehat{P} \subset V$
the tautological line over $P$ and by $T_P(\widehat{X}) \subset V$ the affine tangent space, the $\C^*$-action $E_P$ has weight $0$ on $\widehat{P}$, weight $1$ on $T_P(\widehat{X})/\widehat{P}$ and weight 2 on
$V/T_P(\widehat{X})$. In other words, there exist  subspaces $T'_P \subset T_P(\widehat{X})$ and $N'_P \subset V$ with
$$V = \widehat{P} \oplus T'_P \oplus N'_P  \  {\text and}\ T_P(\widehat{X}) = \widehat{P} \oplus T_P'$$ such that $\lambda \in \cit^*$ acts by $$E_P^{\lambda} (t \in \widehat{P}, u \in T'_P, w \in N'_P) = (t, \lambda u, \lambda^2 w).$$ \item[(2)] The induced $\cit^*$-action on $\BP V$ preserves $X$. \end{itemize}
\end{definition}

\begin{notation} For a projective manifold $M \subset \BP^N,$ the second fundamental form at a point $x \in M$ is the non-empty linear system of quadrics on the tangent space $T_xM$, denoted by
     $|{\rm II}_{x, M}| \subset \pit(\Sym^2 T_xM)^*,$ which comes from  the derivative of the Gauss map (see  Section 3.2 of \cite{IL} or p.602 of \cite{R2} for details).
\end{notation}

\begin{proposition}\label{p.surject}
Let $X \subset \BP V, \dim V = N+1,$ be a quadratically symmetric variety of dimension $n$.  For a fixed $P \in X^o$ and the representation $E_P$ of Definition \ref{d.QSV},  let
$(x_0, x_1, \ldots, x_N)$ be linear coordinates on $V$ such that the dual action of $\lambda \in \C^*$,  denoted by the same symbol $E_P^{\lambda}$ for simplicity, is given by $E_P^{\lambda}(x_0) = x_0$ and $$ E_P^{\lambda} (x_i) = \left\{
\begin{array}{ll} \lambda x_i  & \mbox{ if $ 1 \leq i \leq n$} \\
 \lambda^2 x_i & \mbox{ if $ n+1 \leq i \leq N.$} \end{array}
\right .  $$
Let $z_i= \frac{x_i}{x_0}, 1 \leq i \leq N,$ be the inhomogeneous coordinates on $\BP V$ centered at $P$.
View the germ of $X \subset \BP V$ at $P$ as the graph of the germ of a holomorphic map
$$z_k = F^k(z_1, \ldots, z_n), n+1 \leq k \leq N.$$ Then $F^k(z_1, \ldots, z_n)$ are homogeneous quadratic polynomials and
the system of quadrics  on $T_P X$ defined by $\{ F^k, n+1 \leq k \leq N \}$ under the identification
of $T^*_P X = \C z_1 + \cdots + \C z_n$ is isomorphic to the second fundamental form $| {\rm II}_{P,X}|$ of $X$ at $P$.
\end{proposition}

\begin{proof}
It is well-known (e.g. p. 108 of \cite{IL}) that the quadratic terms in the Taylor expansion of $F^k$'s give the second fundamental form $|{\rm II}_{P,X}|$.
So it suffices to show that $F^k$'s are quadratic.
The $\cit^*$-action $E_P$ induces a $\cit^*$-action on the affine space $\BP V \setminus (x_0 =0) \ = \ \{(z_1, \ldots, z_N) \}$ by
$$ E_P^{\lambda} (z_i) = \left\{
\begin{array}{ll} \lambda x_i  & \mbox{ if $ 1 \leq i \leq n$} \\
 \lambda^2 z_i & \mbox{ if $ n+1 \leq i \leq N.$} \end{array}
\right .  $$
Since this action preserves the germ of $X$ at $P$, the equations $$z_k = F^k(z_1, \ldots, z_n), n+1 \leq k \leq N$$ must remain unchanged
   under the $\cit^*$-action. Thus the Taylor expansion of $F^k$ can have only quadratic terms.
\end{proof}

\begin{proposition}\label{p.II}
Let $X_1$ and $X_2$ be two  quadratically symmetric varieties of the same dimension in
$\BP V$. Assume that there are points $P_1 \in X^o_1$ and
$P_2 \in X^o_2$ such that the second fundamental forms
${\rm II}_{P_1, X_1}$ and ${\rm II}_{P_2, X_2}$ are isomorphic as systems of
quadrics. Then $X_1$ and $X_2$ are
projectively equivalent.\end{proposition}

\begin{proof} We can choose the inhomogeneous coordinates
in Proposition \ref{p.surject} such that the germ of $X_i \subset \BP V$ at $P_i$ for $i=1,2$ is defined by equations $z_k = F_i^k
 (z_1, \ldots, z_n)$ for some quadratic polynomials $F^k_i, n+1 \leq m, i=1,2$. Since ${\rm II}_{P_1, X_1}$ and ${\rm II}_{P_2, X_2}$ are isomorphic,
 we can assume that $F^k_1 = F^k_2$ by linear coordinate changes. It follows that $X_1$ and $X_2$ are projectively equivalent. \end{proof}

The following is a generalization of $Z(\sigma)$ in Definition \ref{d.quadrics}.

\begin{notation}\label{n.sigma}
Let $U, W$ be two vector spaces and let $\sigma: \SYM^2 U \to W$ be a surjective homomorphism.
Fix a 1-dimensional vector space $T$ with a fixed identification $T = \cit$.
 Define a rational map $\phi^o: \BP (T \oplus U) \dasharrow \BP (T \oplus U \oplus W)$
by $$ [t: u] \mapsto  [t^2: tu: \sigma(u,u)] \mbox{ for } t \in T, u \in U.$$ The proper image of $\BP (T \oplus U)$ under $\phi^o$ will be denoted by
$Z= Z(\sigma) \subset \BP (T \oplus U \oplus W)$.
Note that $\phi$ sends $U \cong \BP (T \oplus U) \setminus \BP U$ isomorphically to $Z \setminus \BP (U \oplus W).$ Thus $Z$ is a rational variety and $\phi^o: \BP (T \oplus U) \dasharrow Z$ is a birational map. \end{notation}

\begin{proposition}\label{p.Zqsv}
In Notation \ref{n.sigma}, the variety $Z(\sigma)$ is a quadratically symmetric variety.
\end{proposition}

\begin{proof}
Consider the $\cit^*$-action on $\pit (T \oplus U \oplus W)$ given by $\lambda \cdot [t:u:w] = [t: \lambda u:\lambda^2 w]$ for all $\lambda \in \cit^*.$
Then it preserves $Z(\sigma)$. This gives $E_P$ in Definition \ref{d.QSV} when $P = \BP T \in Z(\sigma)$.

As in the proof of Proposition \ref{p.hyperplane}, we can
associate to each vector $v \in U$ the linear automorphism $g_v$ of $\BP (T \oplus U \oplus W)$ defined by
 $$g_v: [t: u: w] \mapsto [t: u + tv: w + 2 \sigma(u,v) +  t\sigma(v, v)].$$
Then $g_v$ preserves $Z(\sigma)$ as we saw in the proof of Proposition \ref{p.hyperplane}.
In particular, the linear automorphism group of $Z(\sigma) \subset \BP (T \oplus U \oplus W)$ acts transitively on $Z(\sigma) \setminus \BP (U \oplus W).$
Thus $E_P$ exists for any $P \in Z(\sigma) \setminus \BP (U \oplus W).$
 \end{proof}

\begin{proposition} \label{p.Landsberg}
Any quadratically symmetric variety $X \subset \BP V$ is projectively equivalent to  $Z(\sigma)$, where $\sigma: \SYM^2 T_P X \to (T_P \BP V)/T_P X$ is the dual of the second fundamental form of $X$ at a point $P \in X^o$.
\end{proposition}

\begin{proof}
 By Section 3 of \cite{L}, the second fundamental form of $Z(\sigma)$ at a point in $Z(\sigma) \setminus \BP (U \oplus W)$
is isomorphic to $\sigma$. Hence $X$ is projectively equivalent to $Z(\sigma)$ by Proposition \ref{p.II} and Proposition \ref{p.Zqsv}.
\end{proof}

Now we can given an intrinsic characterize the projective manifold $Z(\sigma)$ associated to a special system   of quadrics $\sigma$
as follows.

\begin{theorem}\label{t.qsv}
Let $X \subset \BP V$ be  a prime Fano manifold. Then $X \subset \BP V$ is a quadratically symmetric variety if and only if $X = Z(\sigma)$ for a special system of quadrics $\sigma$. In particular, Theorem \ref{t.bir} gives a classification of quadratically symmetric prime Fano manifolds.
\end{theorem}

\begin{proof}
From Theorem \ref{t.bir} and Proposition \ref{p.Zqsv}, the manifold $Z(\sigma)$ associated to a special system of quadrics $\sigma$  is
a quadratically symmetric prime Fano manifold.

Conversely, a quadratically symmetric prime Fano manifold is of
the form $Z(\sigma)$ by Proposition \ref{p.Landsberg} where
$\sigma$ is its second fundamental form. By Corollary 2.15 (iv) of
\cite{AS2}, the locus $\sigma(u, u)=0$ is exactly the VMRT at a
general point of $X$. Since $Z(\sigma)$ is a quadratically
symmetric prime Fano manifold, there are Euler vector fields
through general points of $Z(\sigma)$, hence its VMRT at a
 general point is smooth and irreducible by Proposition 6.6 of \cite{FH}.
This implies that $X \subset \pit V$
is  $Z(\sigma)$ associated to a special system of quadratics $\sigma$. \end{proof}

Theorem \ref{t.qsv} enables us to relate special quadratic manifolds to the problem studied in \cite{FH}.
Let us recall the basic definitions.

\begin{definition} \label{d.prolongation}
Let $V$ be  a complex vector space and $\fg \subset {\rm End}(V)$
a Lie subalgebra. The {\em first prolongation} (denoted by
$\fg^{(1)}$) of $\fg$ is the space of bilinear
homomorphisms $A: \Sym^{2}V \to V$ such that for any fixed $w \in V$, the endomorphism $A_{w}: V \to
V$ defined by $$v\in V \mapsto A_{w,v} := A(w,v) \in V$$ is in $\fg$. In other words, $\fg^{(1)}
= \Hom(\Sym^{2}V, V) \cap \Hom(V, \fg)$.
\end{definition}

\begin{definition}\label{d.basic}
Let $X \subset \BP V$ be an irreducible subvariety.
Denote by $\widehat{X} \subset V$ the affine cone of $X$
and by $T_{\alpha}(\widehat{X}) \subset V$ the tangent space at a
smooth point $\alpha \in \widehat{X}$. The Lie algebra of
infinitesimal linear automorphisms of $\widehat{X}$ is
$$\aut(\widehat{X}):=\{g \in \End(V)| g(\alpha) \in T_{\alpha}(\widehat{X})
\mbox{ for any smooth point } \alpha \in \hat{X} \}.$$ We will call
  $\aut(\widehat{X})^{(1)}$  the  {\em
prolongation} of $X \subset \BP V$. We say that $X$ {\em has nonzero prolongation}
if $\aut(\widehat{X})^{(1)} \neq 0$.
\end{definition}

\begin{proposition}\label{p.prolongation}
Let $X \subset \pit V$ be a quadratically symmetric variety. Then $\aut(\widehat{X})^{(1)} \neq 0$.
\end{proposition}
\begin{proof}
By Proposition \ref{p.Landsberg}, we may assume that $X = Z(\sigma) \subset \pit V$,
with $V = T \oplus U \oplus W$ associated to a surjective homomorphism $\sigma: \SYM^2 U \to W$.
Define $A: {\rm Sym}^2 V \to V$ by
$$
A\left( (t_1, u_1, w_1), (t_2, u_2, w_2) \right) = \left(t_1 t_2, \frac{t_1 u_2 +t_2u_1}{2}, \sigma(u_1, u_2) \right).
$$ We claim that  $A \in \aut(\widehat{X})^{(1)}$ which proves the proposition.

To check the claim, we have to show that  $A(v, \alpha) \in T_\alpha \widehat{X}$ for any $v \in V$ and a general smooth point $\alpha$ of $\widehat{X}$. It suffices to consider $\alpha$ in the open subset $\{(t, u, w)  \mid t \neq 0,  t w = \sigma(u, u)\}$ of $\widehat{X}$.
Fix $$\alpha = (t, u, w=\frac{1}{t}\sigma(u,u)) \in V, t \neq 0.$$ Then $T_{\alpha} \widehat{X} $ is the subspace of $V$ consisting of $(t',u',w')$ satisfying
$$(t+ \epsilon t') (w + \epsilon w') = \sigma (u + \epsilon u', u + \epsilon u') \ \mbox{ modulo } \epsilon^2.$$
Thus \begin{eqnarray*} T_{\alpha} \widehat{X} &=& \{(t', u', w') \in V \mid t w' + t' w - 2 \sigma( u, u') =0\} \\
&=& \{(t',u',w') \mid t w' + \sigma(u, \frac{t'}{t}u - 2 u') =0 \}.\end{eqnarray*}
For any $v = (t_0, u_0, w_0)$, we have
$$A(v, \alpha) = A((t_0, u_0, w_0), (t, u, w)) = (t_0  t, \frac{t_0 u + t u_0}{2}, \sigma(u_0, u)).$$
Writing the right hand side as $(t', u', w')$, we have
$$t w' + \sigma(u, \frac{t'}{t}u - 2 u')  = t \sigma(u_0, u) +  \sigma( u, \frac{t_0 t}{t}  u - 2 \frac{t_0  u + t u_0}{2} ) = 0.$$
This proves the claim. \end{proof}

The following is a partial converse of Proposition \ref{p.prolongation}.

\begin{proposition}\label{p.converse}
 Let $X \subset \BP V$ be a nonsingular nondegenerate linearly normal projective variety.
If $X$ has nonzero prolongation, then   $X$ is a
quadratically symmetric variety.\end{proposition}

\begin{proof}
This is contained in the proof of
Theorem 1.1.3 \cite{HM05}.  In fact, it is shown  in p. 606 of \cite{HM05} that for a nonzero element $A: \Sym^2 V \to V$ of the prolongation of $X$ and a general point $\alpha \in \widehat{X}$, the
endomorphism $A_{\alpha}$ of $V$ satisfies $A_{\alpha}^3 =0$ and the semisimple part of the endomorphism $A_{\alpha}$  generates a $\cit^*$-action with weight 0 on $\alpha$ and weight 1 on $T_{\alpha}X$.
The orbits of this $\cit^*$-action on $X$ have degree $\leq 2$ from   $A_{\alpha}^3 =0$. Since $X$ is nondegenerate in $\BP V$, this implies that the
weight on the complement of $T_{\alpha} \widehat{X}$ is 2. Thus $X$ must be quadratically symmetric. \end{proof}

Using the above results, we can derive from Theorem \ref{t.bir} the following classification result.

\begin{theorem}\label{t.Main}
Let $X \subsetneq \BP V$ be an irreducible nonsingular nondegenerate variety such that $\aut(\widehat{X})^{(1)} \neq 0$. Then
  $X \subset \BP V$ is projectively equivalent to one in the following list.
\begin{itemize} \item[(1)] The VMRT of an irreducible Hermitian symmetric space of compact type  of rank $\geq 2$ from Example \ref{e.IHSS}.
\item[(2)] The VMRT of a symplectic Grassmannian from Example \ref{e.SymGr}.
\item[(3)] A nonsingular linear section of $\Gr(2, 5) \subset \BP^9$ of codimension $\leq 2$.
\item[(4)] A $\BP^4$-general linear section of $\mathbb{S}_5 \subset \BP^{10}$ of codimension $\leq 3$.
\item[(5)] Biregular projections of (1) and (2) with nonzero prolongation, which are completely described in Section 4 of \cite{FH}.
\end{itemize}  \end{theorem}
\begin{proof}
The assumption $\aut(\widehat{X})^{(1)} \neq 0$ implies that $X$ is conic-connected by Theorem 1.1.3 \cite{HM05}.

If $X$ is not a prime Fano manifold and $X$ is linearly normal, then  $X \subset \BP V$ can be classified by Theorem \ref{t.IR}. This is done
in Proposition 6.4 of \cite{FH}: they are exactly (1) and (2), excepting the prime Fano manifolds belonging to (1).
Note that prime Fano manifolds belonging to (1) are exactly those appearing as $Z(\sigma)$ in Example \ref{e.table}.

 If $X$ is a prime Fano manifold and $X \subset \BP V$ is linearly normal, then
 by  Proposition \ref{p.Landsberg} and Proposition \ref{p.converse}, we see that $X \subset \pit V$ is projectively equivalent to $Z(\sigma)$, where $\sigma$ is the second fundamental form of $X$ at a general point.
Thus we have (3), (4) and prime Fano manifolds in (1),  from Theorem \ref{t.bir} and the data on (1) and (2) in Section 3 of \cite{FH}.

Finally, if $X$ is not linearly normal, it must be a biregular projection of the linearly normal one by Corollary 4.8 \cite{FH}.
(3) and (4) do not have biregular projections. So we have (5).
\end{proof}

\begin{remark}\label{r.FH}
Main Theorem in \cite{FH} asserted the classification result in Theorem \ref{t.Main}, but
there was a flaw: the list there missed the case of (3) with codimension 2 and the cases of (4) with codimension 2 and 3.
This omission is caused by  Proposition 2.9 of \cite{FH}, where it was claimed that $\dim \aut(\hat{X})^{(1)} \neq 1$ for any nonsingular nondegenerate linearly normal variety $X \subset \BP V$.
This proposition is incorrect: the codimension-2 case of (3)  and the codimension-3 case of (4)
 satisfy  $\dim \aut(\hat{X})^{(1)} = 1$.    The error in the proof
  of Proposition 2.9 occurred at the very last line: the statement that the set of
$\alpha$ and $\alpha'$ (in the notation therein) spans the vector space $V$ is wrong.
\end{remark}

\bigskip
Baohua Fu

Institute of Mathematics, AMSS, Chinese Academy of Sciences,

55 ZhongGuanCun East Road, Beijing, 100190, China

 bhfu@math.ac.cn

\bigskip
Jun-Muk Hwang

 Korea Institute for Advanced Study, Hoegiro 85,

Seoul, 130-722, Korea

jmhwang@kias.re.kr

\end{document}